\documentclass{amsart}
\usepackage{amsmath, amscd, amsfonts, latexsym, amsthm, amsxtra,  amssymb, amscd, rotating, epsfig}
\usepackage{bbm}
\usepackage[all]{xy}

\allowdisplaybreaks
\newtheorem{thm}{Theorem}[section]
\newtheorem{pro}[thm]{Proposition}
\newtheorem{lm}[thm]{Lemma}
\newtheorem{cor}[thm]{Corollary}
\newtheorem{wh}[thm]{Working Hypothesis}
\newtheorem{asr}[thm]{Assumption}
\numberwithin{equation}{section}

\theoremstyle{remark}
\newtheorem{exa}[thm]{Example}
\newtheorem{rem}[thm]{Remark}

\theoremstyle{definition}
\newtheorem{defn}[thm]{Definition}

\DeclareMathOperator*{\reg}{reg}

\DeclareMathOperator*{\scn}{sc}
\DeclareMathOperator*{\Irr}{Irr}

\DeclareMathOperator*{\Gal}{Gal}

\DeclareMathOperator*{\ad}{ad}

\DeclareMathOperator*{\Nrd}{Nrd}
\DeclareMathOperator*{\der}{der}

\DeclareMathOperator*{\Hom}{Hom}

\newcommand{\s}{\simeq}
\newcommand{\pp}{\mathfrak{p}}

\newcommand{\OO}{\mathcal{O}}

\newcommand{\vphi}{\varphi}
\newcommand{\ma}{\mathfrak{a}}
\newcommand{\St}{\mathsf{St}}

\newcommand{\tM}{\widetilde{M}}
\newcommand{\tsigma}{\widetilde{\sigma}}
\newcommand{\e}{\mathcal{E}_{u}^{\circ}}
\newcommand{\ed}{\mathcal{E}_{u}^{2}}
\newcommand{\esq}{\mathcal{E}^{2}}

\newcommand{\tlsigma}{\tilde{\sigma}}

\newcommand{\tlchi}{\tilde{\chi}}

\renewcommand{\AA}{\mathbb{A}}
\newcommand{\CC}{\mathbb{C}}

\newcommand{\RR}{\mathbb{R}}
\newcommand{\QQ}{\mathbb{Q}}
\newcommand{\ZZ}{\mathbb{Z}}

\newcommand{\M}{\mathbf{M}}
\newcommand{\Z}{\mathbf{Z}}
\newcommand{\F}{\mathbf{F}}
\newcommand{\D}{\mathbf{D}}
\newcommand{\G}{\mathbf{G}}
\newcommand{\wG}{\widetilde{\mathbf{G}}}
\newcommand{\wM}{\widetilde{\mathbf{M}}}

\newcommand{\ii}{\mathbf{\textit{i}}}

\newcommand{\dprod}{\displaystyle\prod}
\newcommand{\dint}{\displaystyle \int \nolimits}
\newcommand{\bG}{G}

\newcommand{\bA}{\mathbf A}

\newcommand{\bM}{M}
\newcommand{\tG}{\widetilde{G}}

\author{Kwangho Choiy}
\address{
Mathematics Department\\
Purdue University\\
West Lafayette, IN 47907\\
U.S.A.}
\email{kchoiy@math.purdue.edu}

\address{\textit{Current Address} : 
Department of Mathematics\\
Oklahoma State University\\
Stillwater, OK 74078-1058\\
U.S.A.}
\email{kwangho.choiy@okstate.edu}
\begin{document}
\title[Transfer of Plancherel Measures between {$p$}-adic Inner Forms]{Transfer of Plancherel Measures for Unitary Supercuspidal Representations between {$p$}-adic Inner Forms}

\maketitle

\begin{abstract}
Let $F$ be a $p$-adic field of characteristic $0$, and let $M$ be an $F$-Levi subgroup of a connected reductive $F$-split group such that $\Pi_{i=1}^{r} SL_{n_i} \subseteq M \subseteq \Pi_{i=1}^{r} GL_{n_i}$ for positive integers $r$ and $n_i$. We prove that the Plancherel measure for any unitary supercuspidal representation of $M(F)$ is identically transferred under \textit{the local Jacquet-Langlands type correspondence} between $M$ and its $F$-inner forms, assuming a working hypothesis that Plancherel measures are invariant on a certain set. This work extends the result of Mui{\'c} and Savin (2000) for Siegel Levi subgroups of the groups $SO_{4n}$ and $Sp_{4n}$ under the local Jacquet-Langlands correspondence. It can be applied to a simply connected simple $F$-group of type $E_6$ or $E_7$, and a connected reductive $F$-group of type $A_{n}$, $B_{n}$, $C_n$ or $D_n$.
\end{abstract}

\section{Introduction} \label{intro}

The theory of Plancherel measures is well known for real groups \cite{arthur89f, hc76, ks82}. Arthur gave an explicit formula of the Plancherel measure in terms of Artin factors \cite{arthur89f}. It follows from the local Langlands correspondence for real groups \cite{lan89} that Plancherel measures are invariant on $L$-packets, and they are preserved by inner forms. Plancherel measures thus turn out to be identical if they are associated to the same $L$-parameter.

For $p$-adic groups, however, the behavior of the Plancherel measure is not completely understood. Although the Lefschetz principle \cite{hc162} conjectures that what is true for real groups is also true for $p$-adic groups, the $L$-packet invariance of the Plancherel measure is currently known only for some cases \cite{keysh, sh90, gtsp10, gt, gtan12, art11, gs12com, choiy2}.
Also, it was proved that the Plancherel measures are preserved by $p$-adic inner forms for the following cases in characteristic $0$.
Arthur and Clozel proved the argument for discrete series representations under the local Jacquet-Langlands correspondence between $GL_n$ and its inner forms \cite{ac89}. The author also verified it for depth-zero supercuspidal representations associated to tempered and tame regular semi-simple elliptic $L$-parameters with an unramified central character between an unramified group and its inner forms \cite{choiy2}. The approaches in both of two results are based on $p$-adic harmonic analysis. For the groups $SO_{4n}$ and $Sp_{4n}$, using a local to global argument, Mui{\'c} and Savin proved that Plancherel measures for unitary supercuspidal representations are preserved under the local Jacquet-Langlands correspondence between the Siegel Levi subgroup and its inner forms \cite{ms00}. In a similar way, Gan and Tantono identically transferred Plancherel measures attached to supercuspidal representations having the same $L$-parameter from the Levi subgroup $GL_r \times GSp_4$ of $GSp_{2r+4}$ to its inner forms \cite{gtan12}. 

The purpose of this paper is to prove that Plancherel measures attached to unitary supercuspidal representations are preserved \textit{under the local Jacquet-Langlands type correspondence} (defined below), assuming a working hypothesis that Plancherel measures are invariant on a certain set. To be precise, let $F$ denote a $p$-adic field of characteristic $0$, and let $M$ be an $F$-Levi subgroup of a connected reductive $F$-split group $G$ such that
\begin{equation} \label{cond on M}
\Pi_{i=1}^{r} SL_{n_i} \subseteq M \subseteq \Pi_{i=1}^{r} GL_{n_i} 
\end{equation}
for positive integers $r$ and $n_i$. Let $G'$ be an $F$-inner form of $G$, and let $M'$ be an $F$-Levi subgroup of $G'$ that is an $F$-inner form of $M$. Then, $M'$ satisfies the following property
\[
\Pi_{i=1}^{r} SL_{m_i}(D_{d_i}) \subseteq M'(F) \subseteq \Pi_{i=1}^{r} GL_{m_i}(D_{d_i}).
\]
Here $D_{d_i}$ denotes a central division algebra of dimension $d_i^2$ over $F$ where $n_i = m_i d_i$. Set $\tM(F) = \Pi_{i=1}^{r} GL_{n_i}(F)$ and $\tM'(F) = \Pi_{i=1}^{r} GL_{m_i}(D_{d_i})$. Denote by $\e(H(F))$ the set of equivalence classes of irreducible unitary supercuspidal representations of $H(F)$ for any algebraic $F$-group $H$. 

Given $\tau \in \e(M(F))$, we have $\tsigma \in \e(\tM(F))$ such that $\tau$ is isomorphic to an irreducible constituent of the restriction $\tsigma|_{M(F)}$. Denote by $\Pi_{\tsigma}(M(F))$ the set of equivalence classes of all irreducible constituents of $\tsigma|_{M(F)}$. Note that the set $\Pi_{\tsigma}(M(F))$ does not depend on the choice of $\tsigma$ and 
is contained in $\e(M(F))$.
On the other hand, we have a unique $\tsigma' \in \e(\tM'(F))$ corresponding to $\tsigma$ under the local Jacquet-Langlands correspondence. Let $\Pi_{\tsigma'}(M'(F))$ have the corresponding meaning for the $F$-inner form $M'$. 
We say that any two representations $\sigma \in \Pi_{\tsigma}(M(F))$ and $\sigma' \in \Pi_{\tsigma'}(M'(F))$ are \textit{under the local Jacquet-Langlands type (JL-type) correspondence}. This will be defined more generally in Definition \ref{def of LJL-type}.

Fix a representative $w \in G(F)$ of a Weyl element $\tilde{w}$ such that $\tilde{w}(\theta) \subseteq \Delta$. Here $\Delta$ denotes the set of simple roots of $A_0$ in $G$, $A_0$ denotes the split component of a minimal $F$-Levi subgroup $M_0$ of $G$, and $\theta$ denotes the subset of $\Delta$ such that $M=M_{\theta}$. We denote by $\mathfrak{a}^{*}_{M, \CC}$ the complex dual of the real Lie algebra of the split component $A_M$ of $M$. Given an irreducible admissible representation $\sigma$ of $M(F)$, $\nu \in \mathfrak{a}^{*}_{M, \CC}$ and $\tilde{w}$, in Section \ref{def of PM}, we define the Plancherel measure as a non-zero complex number $\mu_{M}(\nu, \sigma, w)$ such that
\begin{equation*} \label{def in intro}
A(\nu, \sigma, w) A(w (\nu), w(\sigma), w^{-1})= \mu_{M}(\nu, \sigma, w)^{-1} 
\gamma _{w}(G|P)^{2}.
\end{equation*}
\noindent Let $w'$ and $\mathfrak{a}^{*}_{M', \CC}$ have the corresponding meaning for the $F$-inner form $M'$ (See Section \ref{inner forms} for details). 

\begin{wh} [Working Hypothesis \ref{whyp}] \label{whyp intro}
Let $\sigma'_1$ and $\sigma'_2$ be given in $\Pi_{\tsigma'}(M'(F))$. Then, we have
\[
\mu_{\bM'}(\nu', \sigma'_1, w') = \mu_{\bM'}(\nu', \sigma'_2, w')
\]
for any $\nu' \in \ma^{*}_{\bM', \CC}$.
\end{wh}
When $G'$ is an $F$-inner form of $SL_n$ and $M'$ is any $F$-Levi subgroup of $G'$, this hypothesis is a consequence of Proposition \ref{Compatibility}. It is also proved in \cite{choiy2} for the case when the central character of $\tsigma'$ is unramified and the set $\Pi_{\tsigma'}(M'(F))$ is associated to a tempered and tame regular semi-simple elliptic $L$-parameter. 

Our main result is the following.
\begin{thm} [Theorem \ref{relation bw JL-type and PM}] \label{main thm in intro}
Let $\sigma \in \e(M(F))$ and $\sigma' \in \e(M'(F))$ be under the local JL-type correspondence. Assume that Working Hypothesis \ref{whyp} is valid. Then, we have
\[
\mu_{\bM}(\nu, \sigma, w) = \mu_{\bM'}(\nu, {\sigma}', w')
\]
for any $\nu \in \ma^{*}_{\bM, \CC} \s \ma^{*}_{\bM', \CC}$.
\end{thm}
\noindent As a corollary to Theorem \ref{main thm in intro}, we obtain the invariance of Plancherel measures on the set $\Pi_{\tsigma}(M(F))$. 

Our approach to prove Theorem \ref{main thm in intro} is a general version of the local to global argument by Mui{\'c} and Savin in \cite{ms00}. {First}, given local data $F$, $G$, $G'$, $M$, $M'$, $\tM$ and $\tM'$ above, the following theorem allows us to construct a number field $\F$, $\F$-groups $\G$, $\G'$, $\M$, $\M'$, $\widetilde{\M}$ and $\widetilde{\M}'$ with prescribed local behavior.
\begin{thm} [Theorem \ref{construction of G and G'}] \label{construction of G and G' intro}
Let $F$ be a $p$-adic field of characteristic $0$, let $G$ be a connected reductive quasi-split $F$-group and let $G'$ be its $F$-inner form. Then, there exist a number field $\F$, a non-empty finite set $S$ of finite places of $\F$, a connected reductive quasi-split $\F$-group $\G$ and its $\F$-inner form $\G'$ such that
\begin{itemize}
 \item[(a)] for all $v \in S$, $\F_v \s F$, $\G_v \s  G$, and $\G'_v \s  G'$ over $\F_v$,
 \item[(b)] for all $v \notin S$ including all the archimedean places, $\G_v \s \G'_v$ over $\F_v$,
\end{itemize}
where $\G_v$ and $\G'_v$ denote $\G \times_{\F} \F_{v}$ and $\G' \times_{\F} \F_{v}$, respectively.
\end{thm}
\noindent This theorem is proved by local and global cohomological results in \cite{kot86, pr94}.

{Second}, we find a finite set $V \supseteq S$ and two cuspidal automorphic representations as described in the following Proposition.
\begin{pro}[Proposition \ref{existence pro}] \label{existence pro intro}
Let $\AA$ be the ring of adeles of $\F$. Suppose $\sigma \in \e(M(F))$ and $\sigma' \in \e(M'(F))$ are under the local JL-type correspondence. Then, there exist a finite set $V$ of places of $\F$ containing $S$ and all archimedean places, and two cuspidal automorphic representations $\pi = \otimes_v \pi_v$ of $\M(\AA)$ and $\pi' = \otimes_v \pi'_v$ of $\M'(\AA)$ such that
\begin{itemize}
  \item[(a)] for all $v \in S$, $\pi_v \s \sigma$ and $\pi'_v \in \Pi_{\tsigma'}(M'(F))$,
  \item[(b)] for all  $v \in V - S$, $\pi_v$ and $\pi'_v$ are irreducible constituents of the restriction of an irreducible representation of $\wM(\F_v)$ to $\M(\F_v)$ (note that, for $v \notin S$, $\M_v \s \M'_v$ and $\widetilde{\M}_v \s \widetilde{\M}'_v$ over $\F_v$),
  \item[(c)] for all $v \notin V$, $\pi_v$ and $\pi'_v$ are isomorphic and unramified with respect to $\M(\OO_v)$. Here $\OO_v$ is the ring of integers of $\F_v$.
\end{itemize}
\end{pro}
\noindent We prove Proposition \ref{existence pro intro} by using the result of Henniart \cite{he84}, the global Jacquet-Langlands correspondence \cite{ba08} and the result of Hiraga and Saito \cite{hs11}.

{Next}, we obtain equalities (\ref{identity of intertwining operators of M}) and (\ref{identity of intertwining operators of M'}) in Section \ref{main thms}, which consist of a product of Plancherel measures at places in $V$ and a quotient of the local Langlands $L$-functions for unramified representations. The property (c) of Proposition \ref{existence pro intro} allows us to cancel all local factors appearing outside $V$. For the places in $V - S$, we use the results of Keys and Shahidi in \cite{keysh} and of Arthur \cite{arthur89f}. We thus retain only Plancherel measures appearing inside $S$ as in equation (\ref{equality equality of two pm}). The hypothesis in Theorem \ref{main thm in intro} is needed here to identify all the Plancherel measures attached to $\pi'_v \in \Pi_{\tsigma'}(M'(F))$ for $v \in S$. Finally, from the fact that Plancherel measures are holomorphic and non-negative on the unitary axis, we deduce Theorem \ref{main thm in intro}.

We remark that, our construction in Theorem \ref{construction of G and G' intro} can be applied to any $F$-Levi subgroup $M$ of any connected reductive group over a $p$-adic field of characteristic $0$, including the case when $M \s GL_n$ in \cite{ms00} (cf. Remark \ref{brauer}). Further, Proposition \ref{existence pro intro} extends the global Jacquet-Langlands correspondence for $GL_n$ to a connected reductive $F$-split group $M$ satisfying condition (\ref{cond on M}). 

As applications of Theorem \ref{main thm in intro}, we transfer the reducibility of the induced representations and the edges of complementary series from unitary supercuspidal representations of maximal $F$-Levi subgroups under the local JL-type correspondence. We also prove that the reducibility and the edges are invariant on the set $\Pi_{\tsigma}(M(F))$. Further, our work can be applied to a simply connected simple $F$-group of type $E_6$ or $E_7$, and a connected reductive $F$-group of type $A_{n}$, $B_{n}$, $C_n$ or $D_n$.

This paper is organized as follows. In Section \ref{prelim}, we recall basic notions, terminologies and known results. After reviewing the local and global Jacquet-Langlands correspondences in Section \ref{subs local JL}, we define the local JL-type correspondence in Section \ref{type}. In Section \ref{local to global}, we set up a generalized local to global argument. We prove in Section \ref{main thms} that Plancherel measures are preserved under the local JL-type correspondence, then we present some applications in Section \ref{ex and appl}. In Section \ref{Further Work}, we generalize the result in Section \ref{main thms} to any Levi subgroup under some assumptions. Appendix A gives a few examples of an $F$-Levi subgroup $M$ and its $F$-inner form satisfying condition (\ref{cond on M}).

\textbf{Acknowledgements.}
We are deeply indebted to Professor Freydoon Shahidi for his constant support and profound insights. We are also grateful to Professor Gordan Savin for his valuable advice and corrections to the manuscript. Thanks are due to Professor Dipendra Prasad for his helpful correspondences and comments on a local to global argument.
This work is a part of the author's Ph.D. dissertation at Purdue University.
\section{Preliminaries} \label{prelim}
We recall basic notions, terminologies and known results. We mainly refer to \cite{hai10, hc73, kot86, sh90, shin10, wal03}.
\subsection{Notation and Conventions}
Throughout this paper, $F$ denotes a $p$-adic field of characteristic $0$ and $\F$ denotes a number field unless otherwise stated. Fix algebraic closures $\bar{F}$ and $\bar{\F}$. We shall use the ordinary capital letters $G$, $M$, etc. for groups {defined over a local field} and the boldfaced capital letters $\G$, $\M$, etc. for groups {defined over a global field}.

By abuse of terminology, we identify the set of isomorphism classes with the set of representatives. Given a connected reductive group $G$ over $F$, we use the following notation: $\Irr(G(F))$ denotes the set of isomorphism classes of admissible representations of $G(F)$; $\esq(G(F))$ denotes the set of essentially square-integrable representations in $\Irr(G(F))$; $\Irr_u(G(F))$ denotes the set of unitary representations in $\Irr(G(F))$; $\e(G(F))$ denotes the set of unitary supercuspidal representations in $\Irr_u(G(F))$; and $\ed(G(F))$ denotes the set of discrete series (square-integrable) representations in $\Irr_u(G(F))$.

Denote by $D$ a central division algebra over $F$. We often write $D_d$ to emphasize its dimension $d^2$ over $F$. We denote by $GL_{m}(D)$ the group of all invertible elements of $m \times m$ matrices over $D$, and by $SL_{m}(D)$ the subgroup of elements in $GL_{m}(D)$ whose reduced norm is $1$ (see \cite[Sections 1.4 and 2.3]{pr94}). 
\subsection{Plancherel Measures} \label{def of PM}
Let $G$ be a connected reductive group over a $p$-adic field $F$ of characteristic $0$. Fix a minimal $F$-parabolic subgroup $P_0$ of $G$ with Levi component $M_0$ and unipotent radical $N_0$. Let $A_0$ be the split component of $M_0$, that is, the maximal $F$-split torus in the center of $M_0$. Let $\Delta$ be the set of simple roots of $A_0$ in $N_0$.

Let $P$ be a standard (that is, containing $P_0$) $F$-parabolic subgroup of $G$. Write $P=MN$ with its Levi component $M=M_{\theta} \supseteq M_0$ generated by a subset $\theta \subseteq \Delta$ and its unipotent radical $N \subseteq N_0$. Let $A_M$ be the split component of $M$. Denote by $W_M = W(G, A_M) := N_G(A_M) / Z_G(A_M)$ the Weyl group of $A_M$ in $G$, where $N_G(A_M)$ and $Z_G(A_M)$ are respectively the normalizer and the centralizer of $A_M$ in $G$. For convenience, we write $A_0 = A_{M_0}$ and $W_G = W_{M_0}$.

Denote by $X^{*}(M)_F$ the group of $F$-rational characters of $M$. We denote by $\ma_M := \Hom(X^{*}(M)_F ,\RR) = \Hom(X^{*}(A_M)_F , \RR)$ the real Lie algebra of $A_{M}$. We set the complex dual 
\[
\mathfrak{a}^{*}_{M, \CC} := X^{*}(M)_F \otimes_{\ZZ} \CC.
\]
We define the homomorphism $H_{M} : M(F) \rightarrow \mathfrak{a}_{M}$ by 
\[
q ^{\langle  \chi, H_{M}(m) \rangle} = |\chi(m)|_F
\]
for all $\chi \in X^{*}(M)_F$ and $m \in M(F)$. 
Here $|\cdot|_F$ denotes the normalized absolute value on $F$.
Note that one can extend $H_M$ to $G(F)$ using the Iwasawa decomposition.

For $\sigma \in \Irr(M(F))$ and $\nu \in \ma^{*}_{M, \CC}$, we denote by $I(\nu, \sigma)$ the normalized induced representation 
\begin{equation*} \label{ind}
I(\nu, \sigma) = {Ind}_{P(F)} ^{G(F)} (\sigma \otimes q ^{\langle  \nu, H_{M}(\,) \rangle} \otimes \mathbbm{1}).
\end{equation*}
Here $\mathbbm{1}$ is the trivial representation of $N(F)$.
The space $V(\nu, \sigma)$ of $I(\nu, \sigma)$ consists of locally constant functions $f$ from $G(F)$ into the representation space $\mathcal{H}(\sigma)$ of $\sigma$ such that 
\[
f(mng) = \sigma(m) q ^{\langle  \nu + \rho_P, H_{M}(m) \rangle} f(g),
\]
for $m \in M(F)$, $n \in N(F)$ and $g \in G(F)$. Here $\rho_P$ denotes the half sum of all positive roots in $N$. We often write $\ii_{G,M}(\nu, \sigma)$ for $I(\nu, \sigma)$ in order to specify groups.

We fix a representative $w \in G(F)$ of $\tilde{w} \in W_G$ such that $\tilde{w}(\theta) \subseteq \Delta$. Set $N_{\tilde{w}}:=N_0 \cap wN^{-}w^{-1}$. We fix a Haar measure $dn$ on $N_{\tilde{w}}$. Given $f \in V(\nu, \sigma)$, for $g \in G(F)$, the standard intertwining operator is defined as
\begin{equation*} \label{intertwining}
A(\nu, \sigma, \tilde{w}) f(g) = \int_{N_{\tilde{w}}(F)} f(w^{-1} n g) dn.
\end{equation*}
We set
\begin{equation} \label{hc gamma ft}
\gamma _{\tilde{w}}(G|M) = \int_{\bar{N}_{\tilde{w}}(F)} q ^{\langle 2 \rho_P, H_{M}(\bar{n}) \rangle} d\bar{n},
\end{equation}
where $d\bar{n}$ is a normalized Haar measure on $\bar{N}_{\tilde{w}}(F)$ and $\bar{N}_{\tilde{w}}:=w^{-1}N_{\tilde{w}}w=N^- \cap w^{-1}N_0 w$.

\begin{defn} \label{definition of PM}
Given $\nu \in \ma^{*}_{M, \CC}$, $\sigma \in \Irr_u(M(F))$ and $\tilde{w}(\theta) \subseteq \Delta$, we define \textit{the Plancherel measure} attached to $\nu$, $\sigma$ and $\tilde{w}$ as a non-zero complex number $\mu_{M}(\nu, \sigma, w)$ such that
\[
A(\nu, \sigma, w) A(w (\nu), w(\sigma), w^{-1})= \mu_{M}(\nu, \sigma, w)^{-1} 
\gamma _{w}(G|P)^{2}.
\]
\end{defn}
\begin{rem} \label{useful rmks on pm}
The Plancherel $\mu_{M}(\nu, \sigma, w)$ depends only on $\nu$, $\sigma$ and $\tilde{w}$. It is independent of the choices of any Haar measure and of any representative $w$ of $\tilde{w}$ \cite[p.280]{sh90}. Further, as a function $\nu \mapsto \mu_{M}(\nu, \sigma, w)$ on $\mathfrak{a}^{*}_{M, \CC}$, it extends to a meromorphic function on all of $\mathfrak{a}^{*}_{M, \CC}$. Moreover, it is non-negative and holomorphic on $\sqrt{-1}\ma^{*}_{M}$ \cite[Theorem 20]{hc73}.
\end{rem}

We review two useful properties of the Plancherel measure. Let $\Phi(P,A_M)$ denote the set of reduced roots of $P$ with respect to $A_M$. For $\alpha \in \Phi(P, A_M)$, $A_{\alpha}:=(\ker \alpha \cap A_M)^{\circ}$ denotes the identity component of $(\ker \alpha \cap A_M)$ regarding $\alpha$ as an element in $\ma^{*}_{M}$. Set $M_{\alpha}:=Z_G(A_{\alpha})$ and $P_{\alpha}:= M_{\alpha} \cap P$. Note that $M_{\alpha}$ contains $M$ and $P_{\alpha}$ is a maximal $F$-parabolic subgroup of $M_{\alpha}$ with its Levi decomposition $MN_{\alpha}$ and the split component $A_M$. The Plancherel measure $\mu_{\alpha}(\nu, \sigma, w)$ and the function $\gamma_{\alpha}(P_{\alpha}|M)$ can be defined by replacing $G$ with $P_{\alpha}$ in Definition \ref{definition of PM} and equation (\ref{hc gamma ft}), respectively. The following is the product formula of the rank-one Plancherel measures \cite[Theorem 24]{hc73}. 
\begin{pro}[Product Formula] \label{prod formula}
\[
\gamma_{\tilde{w}}(G|M)^{-2}\mu_{M}(\nu, \sigma, w) 
= \dprod_{\alpha \in \Phi(P,A_M)} \gamma_{\alpha}(P_{\alpha}|M)^{-2} \mu_{\alpha}(\nu, \sigma, w).
\]
\end{pro}

Next, suppose that $\widetilde{G}$ is a connected reductive group over $F$ such that
\[
G_{\der} = \tG_{\der} \subseteq G \subseteq \widetilde{G},
\]
where the subscript $_{\der}$ means the derived group. Let $\widetilde{M}$ denote an $F$-Levi subgroup of $\widetilde{G}$ such that $M = \widetilde{M} \cap G$. Given $\sigma \in \Irr(M(F))$, we have an irreducible representation $\widetilde{\sigma} \in \Irr(\widetilde{M}(F))$ whose restriction $\widetilde{\sigma}|_{M(F)}$ contains $\sigma$ (see \cite[Proposition 2.2]{tad92}). Since $A_M \subseteq A_{\widetilde{M}}$, we have a canonical surjective homomorphism
$
\ma^{*}_{\widetilde{M}, \CC} \twoheadrightarrow \ma^{*}_{M, \CC}.
$
We write $\widetilde{\nu}$ for any pre-image in $\ma^{*}_{\widetilde{M}, \CC}$ of $\nu \in \ma^{*}_{M, \CC}$.
The following is the compatibility of the Plancherel measure with the restriction.
\begin{pro}[Compatibility with Restriction] \label{Compatibility}
\[
\mu_{M}(\nu, \sigma, w) = 
\mu_{\widetilde{M}}(\widetilde{\nu}, \widetilde{\sigma}, w).
\]
\end{pro}
\begin{proof}
Since $G$ and $\widetilde{G}$ have the same derived group, we have 
\[
\gamma_{\tilde{w}}(G|M)^{-2} = \gamma_{\tilde{w}}(\widetilde{G}|\widetilde{M})^{-2}.
\]
Hence, the proposition follows from the following property (cf. \cite[p.293]{sh89}):
\[
A(\nu, \sigma, \tilde{w}) = A(\widetilde{\nu}, \widetilde{\sigma}, \tilde{w}) |_{\ii_{G,M}(\nu, \sigma)}.
\]
\end{proof}
\begin{rem} \label{remark for unitary and max}
Remark \ref{useful rmks on pm}, Propositions \ref{prod formula} and \ref{Compatibility} reduce the study of Plancherel measures to the case of irreducible unitary admissible representations of maximal $F$-Levi subgroups of a semi-simple group $G$.
\end{rem}
\subsection{Inner Forms} \label{inner forms}
Let $G$ and $G'$ be connected reductive groups over a $p$-adic field $F$ of characteristic $0$. We say that $G$ and $G'$ are \textit{$F$-inner forms} with respect to an $\bar{F}$-isomorphism $\varphi: G' \rightarrow G$ if $\varphi \circ \tau(\varphi)^{-1}$ is an inner automorphism ($g \mapsto xgx^{-1}$) defined over $\bar{F}$ for all $\tau \in \Gal (\bar{F} / F)$ (see \cite[p.851]{shin10}). Denote by $Z(G)$ the center of $G$. Set $G^{\ad}:=G / Z(G)$ and $H^i(F, G) := H^i(\Gal (\bar{F} / F), G(\bar{F}))$. We note \cite[p.280]{kot97} that there is a bijection between $H^1(F, G^{\ad})$ and the set of isomorphism classes of $F$-inner forms of $G$, obtained by sending the isomorphism class of a pair ($G'$, $\varphi$) to the class of the $1$-cocycle $\tau \mapsto \varphi \circ \tau(\varphi)^{-1}$. We note that a $\Gal (\bar{F} / F)$-stable $G^{\ad}(\bar{F})$-orbit of $\varphi$ gives the same isomorphism class of a pair ($G'$, $\varphi$). We often omit the reference to $\varphi$ when there is no danger of confusion. We note that this notion holds for any field $F$ with an algebraic separable closure $\bar{F}$.

Suppose that $G$ is \textit{quasi-split} over $F$. Let $G'$ be an $F$-inner form of $G$ with respect to an $\bar{F}$-isomorphism $\varphi : G' \rightarrow G$ ($G'$ can be $G$ itself). Let $P_0$, $M_0$, $N_0$, $A_0$, $P'_0$, $M'_0$, $N'_0$, $A'_0$ and $\Delta'$ be as in Section \ref{def of PM}. We denote by $P'=M' N'$ a standard $F$-parabolic subgroup of $G'$. Then, from \cite[Section 11.2]{hai10} we can choose an element $\varphi_1$ in a $\Gal (\bar{F} / F)$-stable $G^{\ad}(\bar{F})$-orbit of $\varphi$ such that: $\varphi_1(A'_0) \subseteq A_0$; $\varphi_1(P')=P=M N$ is a standard $F$-parabolic subgroup of $G$; and $\varphi_1(M')=M$. 
As we discussed above, we identify ($M'$, $\varphi_1$) with ($M'$, $\varphi$). Hence, $M'$ is the $F$-inner form of quasi-split group $M$ over $F$ with respect to the $\bar{F}$-isomorphism $\varphi :M' \rightarrow M$. In this case, we often say that $M$ and $M'$ are corresponding. In fact, the split components $A_M$ and $A_{M'}$ are isomorphic over $F$ via the $\bar{F}$-isomorphism $\varphi$. We thus have 
$
\ma^{*}_{M, \CC} \s \ma^{*}_{M', \CC}
$ 
as $\CC$-vector spaces and identify $\nu \in \ma^{*}_{M, \CC}$ with $\nu' \in \ma^{*}_{M', \CC}$ through the isomorphism. Denote by $W_M = W(G, A_M)$ and $W_{M'} = W(G', A_{M'})$ the Weyl groups as defined in Section \ref{def of PM}. Let $\tilde{w}' \in W_{G'}$ be given such that $\tilde{w}'(\theta') \subseteq \Delta'$. Then the image $\tilde{w} := \varphi(\tilde{w}')$ in $W_{G}$ satisfies the property that $\tilde{w}(\theta) \subseteq \Delta$ (see \cite[Section 11.2]{hai10}).

We recall the Kottwitz isomorphism. Set $A(G) := \pi_0(Z(\hat{G})^{\Gamma_F})^D$. Here $\hat{G}$ denotes the connected Langlands dual group ($L$-group) of $G$, $Z(\hat{G})$ denotes the center of $\hat{G}$, $\Gamma_F = \Gal(\bar{F}/F)$, $\pi_0( \; \cdot \; )$ denotes the group of connected components, and $( \; \cdot \; )^D$ denotes the Pontryagin dual, that is, $\Hom(\; \cdot \; , \RR/\ZZ)$. We note that $A(G)$ is a finite abelian group when $G=G^{\ad}$. To see this, since the center of a simply-connected semisimple complex Lie group is finite, $Z(\widehat{G^{\ad}}) = Z(\hat{ G}_{\scn})$ turns out to be a finite abelian group. Here $\hat{G}_{\scn}$ denotes the simply connected cover of the derived group of $\hat{ G}$.

\begin{pro} (Kottwitz, \cite[Theorem 1.2]{kot86}) \label{kot iso}
Let $G$ be a connected reductive group over a $p$-adic field $F$ of characteristic $0$. One has the following bijection
\[
H^1(F, G) \s A(G).
\]
\end{pro}
\noindent We note that if the adjoint group $G^{\ad}$ of $G$ is simply connected (e.g., $G_2$, $F_4$, or $E_8$), there is no non-quasi-split $F$-inner form due to Proposition \ref{kot iso}.
\begin{exa} \label{ex for inner forms of GL}
Let $G$ be either $GL_n$ or $SL_n$ over a $p$-adic field $F$ of characteristic $0$. Then the set of isomorphism classes of $F$-inner forms of $G$ is in bijection with the subgroup $Br(F)_n$ of $n$-torsion elements in the Brauer group $Br(F)$. Indeed, we have
\[
H^1(F, PGL_n) \s A(PGL_n) \s \mu_n(\CC)^D.
\]
By Hilbert's theorem 90, $\mu_n(\CC)^D \s H^2(F, \mu_n) \s \ker(Br(F) \overset{n}{\longrightarrow} Br(F)) = Br(F)_n$.
Here $\mu_n$ is the algebraic group of $n^{th}$ root of unity. 
\end{exa}
\begin{exa} \label{ex for inner forms of GL over number field}
Let $\G$ be either $\mathbf{GL}_n$ or $\mathbf{SL}_n$ over a number field $\F$. Then the set of isomorphism classes of $\F$-inner forms of $\G$ is also in bijection with $Br(\F)_n$. Indeed, since $H^1(\F, \mathbf{GL}_n) = 1$ (in fact, this is true for any perfect field \cite[Lemma 2.2 p.70]{pr94}), we have
\[
H^1(\F, \mathbf{PGL}_n) \hookrightarrow H^2(\F, \mathbf{\mu}_n) \s \mu_n(\CC)^D.
\]
The injection turns out to be surjective due to \cite[Theorem 6.20]{pr94}. 
\end{exa}  
\section{Local and Global Jacquet-Langlands Correspondences for $GL_n$} \label{subs local JL}
In this section, we recall the Jacquet-Langlands correspondence for $GL_n$ over a $p$-adic field $F$ of characteristic $0$ and a number field $\F$. We mainly refer to \cite{ba08}.

The local and global Jacquet-Langlands correspondences were initially found by Jacquet and Langlands \cite{jl} for the case $GL_2$ in any characteristic.
The local generalization to $GL_n$ in zero characteristic was established by Rogawski \cite{rog83} and independently by Deligne, Kazhdan and Vigneras \cite{dkv}. Badulescu completed the proof of the local correspondence for $GL_n$ in positive characteristic \cite{ba02}. On the other hand, the global generalization of the correspondence to $GL_n$ was proved only for a number field by Badulescu \cite{ba08}. For some particular cases in zero characteristic, Flath treated the local and global correspondences for $GL_3$ \cite{flath81}, and Snowden presented a new purely local proof of the local Jacquet-Langlands correspondence for $GL_2$  \cite{snow}.

\subsection{Local Jacquet-Langlands Correspondence} \label{local JL}
Let $G$ be $GL_n$ over a $p$-adic field $F$ of characteristic $0$ and let $G'$ be an $F$-inner form of $G$. Then $G'(F)$ is of the form of $GL_m(D)$, where $D$ denotes a central division algebra of dimension $d^{2}$ over $F$ and $n=md$. 

For semisimple elements $g \in G(F)$ and $g' \in G'(F)$, we write $g \leftrightarrow g'$ if both are regular 
(i.e., all roots in $\bar{F}$ of the characteristic polynomial are distinct) and have the same characteristic polynomial. We write $G(F)^{\reg}$ for the set of regular semisimple elements in $G(F)$. We denote by $C^{\infty}_{c}(G(F))$ the Hecke algebra of locally constant functions on $G(F)$ with compact support. Fix a Haar measure $dg$ on $G(F)$. For any $\rho \in \Irr(G(F))$, there is a unique locally constant function $\Theta_{\rho}$ on $G(F)^{\reg}$ which is invariant under conjugation by $G(F)$ such that
\[
tr \rho (f) = \dint_{G(F)^{\reg}} \Theta_{\rho}(g)f(g)dg
\]
for all $f \in C^{\infty}_{c}(G(F))$.
We refer to \cite[p.96]{hc81} and \cite[b. p.33]{dkv} for details. The same is true for the $F$-inner form $G'$. We state the local Jacquet-Langlands correspondence as follows.

\begin{pro} (\cite[B.2.a]{dkv}, \cite[Theorem 5.8]{rog83}, and \cite[Theorem 2.2]{ba08}) \label{proposition of local JL for essential s i}
There is a unique bijection $\mathbf{C} : \esq(G(F)) \longrightarrow \esq(G'(F))$ such that: for all $\sigma \in \esq(G(F))$, we have
\[
\Theta_{\sigma}(g) = (-1)^{n-m} \Theta_{\mathbf{C}(\sigma)}(g') 
\]
for all $g \leftrightarrow g'$. 
\end{pro}
\begin{rem}(\cite[Introduction d.4)]{dkv}) \label{torsion by character}
For any $\sigma \in \esq(G(F))$ and quasi-character $\eta$ on $F^{\times}$, we have $\mathbf{C}(\sigma \otimes (\eta \circ \det)) = \mathbf{C}(\sigma) \otimes (\eta \circ \Nrd)$, where $\Nrd : GL_m(D) \rightarrow F^{\times}$ is the reduced norm (cf. \cite[Section 53.5]{bh06}).
\end{rem}
\begin{exa} \label{example for st}
Denote by $\St_G$ (resp. $\St_{G'}$) the Steinberg representation of $G(F)$ (resp. $G'(F)$). Since $\Theta_{\St_G}(g) = (-1)^{n-m} \Theta_{\St_{G'}}(g')$ for all $g \leftrightarrow g'$ \cite[Section 15]{hc73}, we have $\mathbf{C} (\St_G) = \St_{G'}$.
\end{exa}

We denote by $R(G)$ the Grothendieck group of admissible representations of finite length of $G(F)$. So, $R(G)$ is a free abelian group with basis $\Irr(G(F))$. Let $R(G')$ be the Grothendieck group for the $F$-inner form $G'$. In what follows, we extend the correspondence $\mathbf{C}$ to a $\ZZ$-morphism from $R(G)$ to $R(G')$. We refer to \cite[Section 2.7]{ba08}.

Let $\mathcal{B}$ be the collection of all normalized (twisted by $\delta_{P}^{1/2}$) induced representation $\ii_{G,L} \sigma$, where $L$ is a standard Levi subgroup of $G$ and $\sigma \in \esq(L(F))$. Let $\mathcal{B}'$ have the corresponding meaning for the $F$-inner form $G'$. We notice that any element $\Sigma \in \mathcal{B}$ (resp. $\Sigma' \in \mathcal{B}'$) has a unique irreducible quotient by the Langlands classification (see \cite[Theorem 1.2.5]{ku94}). We denote it by $Lg(\Sigma)$ (resp. $Lg(\Sigma')$).
We note that the set $\mathcal{B}$ is a basis of $R(G)$ and the map $\Sigma \mapsto Lg(\Sigma)$ is a bijection from $\mathcal{B}$ onto $\Irr(G(F))$. The same is true for the $F$-inner form $G'$.

Given a basis element $\Sigma'= \ii_{G',L'} \sigma' \in \mathcal{B}'$ with a standard $F$-Levi subgroup $L'$ of $G'$, we set $\Lambda(\Sigma'):= \ii_{G,L} \mathbf{C}^{-1}(\sigma')$. Here $L$ is the standard Levi subgroup of $G$ corresponding to $L'$ (see Section \ref{inner forms}). From Proposition \ref{proposition of local JL for essential s i}, we notice that $\Lambda(\Sigma')$ lies in $\mathcal{B}$. Thus, $\Lambda$ defines a map from $\mathcal{B}'$ into $\mathcal{B}$, which is clearly injective. Further, since $\Lambda(\Sigma')$ has a unique irreducible quotient (denoted by $Lg(\Lambda(\Sigma'))$), $\Lambda$ induces a map from $\Irr(G'(F))$ into $\Irr(G(F))$ by sending 
$Lg(\Sigma') \mapsto Lg(\Lambda(\Sigma'))$.
 
\begin{defn} \label{def of LJ}
We define a $\ZZ$-morphism $\mathbf{LJ} : R(G) \rightarrow R(G')$
by setting $\mathbf{LJ}(\Lambda(\Sigma'))= \Sigma'$ and $\mathbf{LJ}(\Sigma) = 0$ 
if $\Sigma$ is not in the image of $\Lambda$.
\end{defn} 
\begin{rem}
There exists an irreducible unitary representation of $G'(F)$ which is not in the image of 
$\Lambda$ \cite[Lemma 3.11]{ba08}. Moreover, the map $\mathbf{LJ}$ sends an irreducible unitary representation of $G(F)$ to either $0$ or an irreducible unitary representation of $G'(F)$ \cite[Proposition 3.9]{ba08}.
\end{rem}
We have the following correspondence between the Grothendieck groups.
\begin{thm} (\cite[Theorem 2.7]{ba08}) \label{prop of extened JL}
There is a unique map $\mathbf{LJ} : R(G) \rightarrow R(G')$ such that: for all $\sigma \in R(G)$, we have
\[
\Theta _{\sigma}(g) = (-1)^{n-m} \Theta_{\mathbf{LJ}(\sigma)}(g')
\]
for all $g \leftrightarrow g'$.
Furthermore, $\mathbf{LJ}$ is a surjective group homomorphism.
\end{thm}
\begin{exa} \label{example for LJ}
Suppose $G=GL_2$. Denote by $\mathbbm{1}_G$ the trivial representation of $G(F)$. Fix a Borel subgroup $B=TU$ of $G$. Note that $\ii_{G, T} (\delta_{B}^{-1/2}) = \mathbbm{1}_G + \St_G$ as elements of $R(G)$. Since there is no $F$-Levi subgroup of $G'$ corresponding $T$, we have $\mathbf{LJ}(\ii_{G, T} (\delta_{B}^{-1/2})) = 0$. Thus, Example \ref{example for st} yields 
\[
\mathbf{LJ}(\St_G) = \St_{G'} ~ ~ \text{and} ~ ~ \mathbf{LJ}(\mathbbm{1}_G) = - \St_{G'}.
\]
\end{exa}
\begin{defn} \label{def of d-compatible}
We say that $\sigma \in R(G)$ is $d$\textit{-compatible} if $\mathbf{LJ}(\sigma) \neq 0$.
\end{defn}
\begin{rem} \label{remark for extended d-compatible}
It follows from Proposition \ref{proposition of local JL for essential s i} that $\sigma \in \esq(G(F))$ is always $d$-compatible.
\end{rem}
We have the following correspondence for $d$-compatible irreducible unitary representations.
\begin{pro}(\cite[Proposition 3.9]{ba08})  \label{local JL for unitary rep}
If $u$ is a $d$-compatible irreducible unitary representation of $G(F)$,
then there exists a unique irreducible unitary representation $u'$ of $G'(F)$ and
a unique sign $\epsilon \in \lbrace -1, 1 \rbrace$ such that
\[
\Theta _{u}(g) = \epsilon \, \Theta_{u'}(g') 
\]
for all $g \leftrightarrow g'$.
\end{pro}
\begin{defn} \label{def of |JL|}
By sending $u \mapsto u'$, we define a map
$|\mathbf{LJ}|$ from the set of irreducible unitary $d$-compatible representations of $G(F)$ to the set of irreducible unitary representations of $G'(F)$. 
\end{defn}
\begin{rem} \label{rem of |JL|}
The restriction of $|\mathbf{LJ}|$ to the set $\ed(G(F))$ equals $\mathbf{C}$.
\end{rem}
\begin{exa} 
For the case that $G=GL_2$, the representation $\ii_{G, T} (\delta_{B}^{-1/2})$ is not $d$-compatible due to Example \ref{example for LJ}. We also note that $|\mathbf{LJ}|(\St_G) = |\mathbf{LJ}|(\mathbbm{1}_G) = \St_{G'}$. 
\end{exa}
\begin{rem} \label{remark for product of local GL's}
All statements in this section admit an obvious generalization to the case that $G$ is a product of a general linear groups.
\end{rem}
\subsection{Global Jacquet-Langlands Correspondence} \label{section global JL}
Let $\G$ be $\mathbf{GL}_n$ over a number field $\F$ and let $\G'$ be an $\F$-inner form of $\G$. Then $\G'(F)$ is of the form of $\mathbf{GL}_m(\D)$, where $\D$ denotes a central division algebra of dimension $d^{2}$ over $\F$ and $n=md$.

Set $\bA := M_{m}(\D)$, the $m \times m$ matrix algebra over $\D$. For each place $v$ of $\F$, we have $\bA_{v}=\bA \otimes_{\F} \F_{v} \s M_{m_{v}}(\D_{v})$ for some positive integer $m_{v}$ and some central division algebra $\D_{v}$ of dimension $d^{2}_{v}$ over $\F_{v}$ such that $m_{v} d_{v}=n$. If $d_{v}=1$ at some place $v$, we say that $\bA$ is \textit{split} at $v$. We denote by $S$ the set of places where $\bA$ is not split. Then $S$ turns out to be finite. Set $G'_v := G' \times_{\F} \F_{v}$. We notice that $\G'_v \s \G_v \s GL_n$ over $\F_v$ for all $v \notin S$.

Let $\AA$ be the ring of adeles of $\F$. We identify $\Z(\G) = \Z(\G') =: \Z$, the centres of $\G$ and $\G'$. Fix a unitary smooth character $\omega$ of the quotient $\Z(\F) \backslash \Z(\AA)$. Let $L^2(\Z(\AA) \G'(\F) \backslash \G'(\AA) ; \omega)$ be the space of classes of functions $f : \G'(\AA) \rightarrow \CC$ which are left invariant under $\G'(\F)$, transform under $\Z(\AA)$ by $\omega$ and are square-integrable modulo $\Z(\AA) \G'(\F)$. Consider the representation $R'_{\omega}$ of  $\G'(\AA)$ by the right translation in the space  $L^2(\Z(\AA) \G'(\F) \backslash \G'(\AA) ~ ; ~ \omega)$. Any irreducible subrepresentation of $R'_{\omega}$ is called to be a \textit{discrete series} of $\G'(\AA)$. Denote by $DS'$ (resp. $DS$) the set of all discrete series of $\G'(\AA)$ (resp. $\G(\AA)$). Every $\pi' \in DS'$ admits the restrict tensor product $\otimes_v \pi'_v$. It turns out that each $\pi'_v$ is an irreducible unitary admissible representation of $\G'(\F_v)$. For each place $v$, we use $d_v$-compatible and $|\mathbf{LJ}|_v$ to emphasize the place $v$ in Definitions \ref{def of d-compatible} and \ref{def of |JL|}, respectively.

\begin{defn} \label{def of D-compatible}
Let $\pi=\otimes_v \pi_v$ be in $DS$. We say that $\pi$ is $D$\textit{-compatible} if $\pi_v$ is $d_v$-compatible for all places $v \in S$.
\end{defn}
For all $v \notin S$, we abuse the notation $|\mathbf{LJ}|_v$ for \textit{the identity map} 
from $\Irr_u(\G(\F_v))$ to $\Irr_u(\G'(\F_v))$. The global Jacquet-Langlands correspondence is as follows.
\begin{thm} (\cite[Theorem 5.1]{ba08}) \label{global JL}
There exists a unique injective map $\Phi : DS' \rightarrow DS$ such that:
for all $\pi'=\otimes_v \pi'_v \in DS'$, we have
\[
|\mathbf{LJ}|_v(\Phi(\pi')_v) = \pi'_v.
\]
Moreover, the image of $\Phi$ is exactly the set of $D$-compatible discrete series of $\G(\AA)$.
\end{thm}

\noindent We note that there was an assumption on the set $S$ in \cite{ba08} and it has been removed in \cite{br10}.
Let the space  $L^2_c(\Z(\AA) \G'(\F) \backslash \G'(\AA) ~ ; ~ \omega)$ denote the subspace of all the cusp forms in $L^2(\Z(\AA) \G'(\F) \backslash \G'(\AA) ~ ; ~ \omega)$. It turns out that the space  $L^2_c(\Z(\AA) \G'(\F) \backslash \G'(\AA) ~ ; ~ \omega)$  is stable under $R'_{\omega}$. Any irreducible subrepresentation of $R'_{\omega}$ in $L^2_c(\Z(\AA) \G'(\F) \backslash \G'(\AA) ~ ; ~ \omega)$ is called to be a \textit{cuspidal automorphic representation} of $\G'(\AA)$. The following proposition describes the behavior of cuspidal automorphic representations under the map $\Phi$.

\begin{pro} (\cite[Proposition 5.5 and Corollary A.8]{ba08}) \label{JL for cuspidal}
Let $\pi \in DS$ be a cuspidal automorphic representation of $\G(\AA)$. If $\pi$ is $D$-compatible, then $\Phi^{-1}(\pi)$ is a cuspidal automorphic representation of $\G'(\AA)$.
\end{pro}
\begin{rem}
There exists a cuspidal automorphic representation of $\G'(\AA)$ whose image through the map $\Phi$ is not cuspidal (see \cite[Proposition 5.5.(b)]{ba08}).
\end{rem}
\begin{rem} \label{remark for product of GL's}
All statements in this section admit an obvious generalization to the case that $\G$ is a product of a general linear groups.
\end{rem}
\section{Local Jacquet-Langlands Type Correspondence} \label{type}
In this section, we define the local Jacquet-Langlands type (JL-type) correspondence which is a general version of the local Jacquet-Langlands correspondence. 

\subsection{Restriction of Representations} \label{section for Tadic results}
We recall the results of Tadi{\'c} in \cite{tad92}. Throughout Section \ref{section for Tadic results}, $G$ and $\tG$ denote connected reductive groups over a $p$-adic field $F$ of characteristic $0$, such that
\[
G_{\der} = \tG_{\der} \subseteq G \subseteq \tG,
\]
where the subscript $_{\der}$ means the derived group.

\begin{pro} (\cite[Lemma 2.1 and Proposition 2.2]{tad92}) \label{prop 2.2 tadic}
For any $\sigma \in \Irr(G(F))$, there exists $\tsigma \in \Irr(\tG(F))$ such that $\sigma$ is isomorphic to an irreducible constituent of the restriction $\widetilde {\sigma} |_{G(F)}$ of $\tsigma$ to $G(F)$.
\end{pro}
\noindent 
Given $\sigma \in \Irr(G(F))$, we denote by $\Pi_{\tsigma}(G(F))$ the set of equivalence classes of all irreducible constituents of $\tsigma|_{G(F)}$. 
\begin{rem} (\cite[Proposition 2.7]{tad92}) \label{remark for sc to sc}
Any member in $\Pi_{\tsigma}(G(F))$ is supercuspidal, essentially square-integrable, discrete series or tempered if and only if $\widetilde{\sigma}$ is.
\end{rem}
\begin{rem} (\cite[Corollary 2.5]{tad92}) \label{remark for lifting}
If $\tsigma_1 \in \Irr(\tG(F))$ is another choice of $\tsigma$ in Proposition \ref{prop 2.2 tadic}, then there exist a quasi-character $\eta$ on $\tG(F) / G(F)$ such that $\tsigma_1 \s \tsigma \otimes \eta$.
It turns out that the set $\Pi_{\tsigma}(M(F))$ is finite and independent of the choice of $\tsigma$.
\end{rem}
\begin{rem} \label{rem of l-packet}
Let $G = SL_n$ over $F$ or its $F$-inner form, and let $\tG=GL_n$ over $F$ or its $F$-inner form. It then turns out \cite{gk82, hs11} that any $L$-packet is of the form $\Pi_{\tsigma}(G(F))$ for some $\tsigma \in \Irr(\tG(F))$.
\end{rem}
\subsection{Definition of the Local Jacquet-Langlands Type Correspondence} \label{section of def of LJL-type}
Let $\tG$ be $\prod_{i=1}^{r} GL_{n_i}$ over a $p$-adic field $F$ of characteristic $0$ and let $\tG'$ be an $F$-inner form of $\tG$. Then $G'(F)$ is of the form of $\prod_{i=1}^{r} GL_{m_i}(D_{d_i})$, where $D_{d_i}$ denotes a central division algebra of dimension $d_i^{2}$ over $F$ and $n_i=m_i d_i$. Let $G$ be a connected reductive $F$-split group such that
\begin{equation}  \label{cond on G and G}
G_{\der} = \tG_{\der} \subseteq G \subseteq \tG.
\end{equation}
Let $G'$ be an $F$-inner form of $G$. It follows that $G'_{\der} = \tG'_{\der} \subseteq G' \subseteq \tG'$.

\begin{defn} \label{def of LJL-type}
Given $\sigma \in \Irr(G(F))$ and $\sigma' \in \Irr(G'(F))$, we say that $\sigma$ and $\sigma'$ are 
\textit{under the local Jacquet-Langlands type (JL-type) correspondence}
if there exist $\tsigma \in \Irr(\tG(F))$ and $\tsigma' \in \Irr(\tG'(F))$ such that
\begin{itemize}
  \item[(a)] $\sigma$ and $\sigma'$ are isomorphic to irreducible constituents of the restrictions $\tsigma |_{G(F)}$ and $\tsigma' |_{G'(F)}$, respectively,
  \item[(b)] $\mathbf{LJ}(\tlsigma) = \tlsigma'$.
\end{itemize}
\end{defn}
\begin{exa} \label{an exa for JL-type}
Steinberg representations $\St_G$ and $\St_{G'}$ are always under the local JL-type correspondence. This follows from Example \ref{example for st} and the fact that the restrictions $\St_{\tG} |_{G(F)}$ and $\St_{\tG'} |_{G'(F)}$ are again Steinberg representations.
\end{exa}
\begin{rem} \label{an rem for JL-type}
Given $\tsigma \in \esq(\tG(F))$ and $\eta \in (\tG(F) / G(F))^D$, set $\tsigma' := \mathbf{C} (\tsigma \otimes \eta) \in \esq(\tG'(F))$ (see Proposition \ref{proposition of local JL for essential s i}). 
Then, any $\sigma \in \Pi_{\tsigma}(G(F))$ and $\sigma' \in \Pi_{\tsigma'}(G'(F))$ are under the local JL-type correspondence.
\end{rem}
\begin{rem} \label{L-packet remark}
The local JL-type correspondence can be regarded as a correspondence between $L$-packets of $G(F)$ and $G'(F)$ (cf. \cite[Chapters 11-15]{hs11}). 
\end{rem}
\section{A Local to Global Argument} \label{local to global}

In this section, we set up a local to global argument which will be used in Section \ref{main thms}. This generalizes the method of Mui{\'c} and Savin in \cite{ms00}.
\subsection{Construction of Global Data from Local Data} \label{const from l to g} 
Given local data, we first construct global data with prescribed local behavior.
\begin{lm} \label{construction F0}
Given a $p$-adic field $F$ of characteristic $0$, there exists a number field $\F^{0}$ such that $\F^{0}_{v_{0}} \s F$ for some finite place $v_{0}$ of $\F^{0}$.
\end{lm}
\begin{proof}
Let $F$ be a finite extension of $\QQ_{p}$ for some prime number $p$. By \cite[Corollary p.44]{lang}, we have a number field $\F^{0}$ such that: $\F^{0}$ is dense in $F$; $F=\F^{0} \cdot \QQ_{p}$; and $[F:\QQ_{p}]=[\F^{0}:\QQ]$. Since $[\F^{0}:\QQ] = \sum_{\pp \mid p} [\F^{0}_{\pp}:\QQ_{p}]$, there exists a unique prime $\pp$ of $\F^{0}$ lying over $p$, and thus $\F^{0} \otimes_{\QQ} \QQ_{p} \s \F^{0}_{\pp} \s F$. By taking $v_{0}=\pp$, we complete the proof.
\end{proof}
\begin{lm} \label{existence of infinitely many prime}
For any prime $p$, there exist infinitely many odd primes $q$ such that $p$ splits completely over $\QQ(\sqrt{q^*})$, where $q^*=(-1)^{\frac{q-1}{2}}q$.
\end{lm}

\begin{proof}
This is an easy consequence of Dirichlet's density theorem.
\end{proof}

\begin{pro} \label{construction F and S}
Let $F$ be a $p$-adic field of characteristic $0$. For any positive integer $l$,  there exist a number field $\F$ and a finite set $T$ of finite places of $\F$, with the cardinality $l$, such that $\F_{v} \s F$ for all $v \in T$.
\end{pro}

\begin{proof}
Given $F$, we fix a number field $\F^0$ and a place $v_0$ as defined in Lemma \ref{construction F0}, so that $\F^0_{v_0} \s F$. From Lemma \ref{existence of infinitely many prime}, we have infinitely many odd primes $q$ such that $v_0$ splits completely over $\F^{0}(\sqrt{q^*})$. Denote by $T_0$ the set of all such primes $q$. For any positive integer $l$, we choose a positive integer $r$ such that $2^{r} \geq l$. Pick a subset $\{q_1, \cdots, q_r \} \subseteq T_0$. Set $\F := \F^{0}(\sqrt{q_1^*}, \cdots, \sqrt{q_r^*})$. We note that $[\F:\F^0] = 2^r$. 
Since $v_0$ splits completely over $\F^0(\sqrt{q_i^*})$ for each $i$, we have $\F_v \s F$ for all $v | v_0$. So, the cardinality $| \lbrace v \mid \F_v \s F \rbrace|$ $\geq$ $2^r$ $\geq$ $l$. By taking $T$ to be any subset of $\lbrace v \vert \F_v \s F \rbrace $ with cardinality $l$, we complete the proof.
\end{proof}

The following theorem allows us to construct a number field and connected reductive groups with prescribed local behavior.
\begin{thm} \label{construction of G and G'}
Let $G$ be a connected reductive \textit{quasi-split} group over a $p$-adic field $F$ of characteristic $0$, and let $G'$ be an $F$-inner form of $G$ with respect to an $\bar{F}$-isomorphism $\varphi : G' \rightarrow G$. Then, there exist a number field $\F$, a non-empty finite set $S$ of finite places of $\F$, a connected reductive group $\G$ defined over $\F$ and its $\F$-inner form $\G'$ such that
\begin{itemize}
\item[(a)] for all $v \in S$, $\F_v \s F$, $\G_v \s  G$, and $\G'_v \s  G'$ over $\F_v$,
\item[(b)] for all $v \notin S$ including all the archimedean places, $\G_v \s \G'_v$ over $\F_v$,
\end{itemize}
where $\G_v$ and $\G'_v$ denote $\G \times_{\F} \F_{v}$ and $\G' \times_{\F} \F_{v}$, respectively.
\end{thm}
\begin{proof}
Let $l$ denote a sufficiently large multiple of the cardinality $|A( G^{\ad})|$ of $A( G^{\ad})$ (see Section \ref{inner forms} for the definition of $A( G^{\ad})$). From Proposition \ref{construction F and S}, we obtain a number field $\F$ and a finite set $T$ of places with the cardinality $l$ such that $\F_v \s F$ for all $v \in T$. We note that the set $A( G^{\ad})$ is a finite abelian group.

Next we choose a connected reductive quasi-split group $\G$ defined over $\F$. It follows that $\G_v := \G \times_{\F} \F_{v} \s  G$ over $\F_v$ for all $v \in T$. We note that $A(\G^{ad})$ is a finite abelian group, and there is a surjective homomorphism $\bigoplus_v A(\G^{\ad}_v) \rightarrow A(\G^{ad})$ for any place $v$ of $\F$ (see \cite[2.3]{kot86}). Here $\bigoplus_v$ denotes the subset of the direct product consisting of $(x_v)$ such that $x_v=1$ for all but a finite number of $v$. Since the integer $l = |T|$ is taken to be sufficiently large, we can assume that the cardinality $| A(\G^{\ad})|$ is smaller than $l$ due to the surjective homomorphism. Choose a subset $S$ of $T$ such that the cardinality $|S|$ equals a multiple of $| A(\G^{\ad})|$.

We recall the following lemma from \cite[Theorem 2.2 and Proposition 2.6]{kot86} and \cite[Theorem 6.22]{pr94}.
\begin{lm}  \label{ko86}
Let $\bar{\AA}$ denote the adele ring of $\bar{\F}$. Then, there is an exact sequence with a commutative diagram:
\[
\begin{CD}
1 @>>> H^1(\F, \G^{\ad}(\bar{\F})) @>\iota_\G>> H^1(\F, \G^{\ad}(\bar{\AA})) 
@>\beta_\G>>A(\G^{\ad}) @>>> 1 \\
@.      @.          @V \thickapprox VV   @A \Sigma AA  @.\\
@.      @.          \bigoplus_{v} H^1(\F_v,\G_v^{\ad}(\bar{\F}_v)) @>>> \bigoplus_v A(\G^{\ad}_v) 
\end{CD}
\]
\end{lm}
\noindent The bottom map is given by local maps defined in Proposition \ref{kot iso}, which are isomorphisms for all finite places of $\F$. Thus, the commutative diagram implies that the morphism $\beta_\G$ is equal to sum of local $1$-cocycles.

Since $G'$ is an $F$-inner form of $G$ with respect to an $\bar{F}$-isomorphism $\varphi : G' \rightarrow G$, we have a $1$-cocycle 
\[
\varphi _{\tau_v} \in H^1(\F_v,  \G^{\ad}_v(\bar{\F}_v))
\]
such that
$\varphi _{\tau_v} := \varphi \circ  {\tau_v}(\varphi)^{-1}$ for $\tau_v \in \Gal(\bar{\F}_v/\F_v)$. Let  
\begin{equation} \label{a_tau}
a_{\tau} :=(a_{\tau_v}) \in H^1(\F, \G^{\ad}(\bar{\AA}))
\end{equation}
be a nontrivial 1-cocycle such that 
$a_{\tau_v}=\vphi_{\tau_v}$ for all $v \in S$ and 
$a_{\tau_v}=1$ for all $v \notin S$. 
Since $A(\G^{\ad})$ is a finite abelian group and $|S|$ is a multiple of $|A(\G^{\ad})|$, we get
\[ 
\beta_\G(a_\tau)=\sum_{v}a_{\tau_v}= |S| \cdot \vphi_{\tau_v}=1.  
\]
From the exactness in Lemma \ref{ko86}, 
we have a nontrivial $1$-cocycle $b_\tau \in H^1(\F, \G^{\ad}(\bar{\F}))$ such that 
$\iota_\G (b_\tau)=a_{\tau}$. 
Thus, we obtain an $\F$-inner form $\G'$ of $\G$ associated to $b_\tau$. From the definition of $a_{\tau}$ in (\ref{a_tau}), it follows that $\G'_v \s  G'$ over $\F_v$ for all $ v \in S$, and $\G_v \s \G'_v$ over $\F_v$ for all $v \notin S$. This completes the proof of Theorem \ref{construction of G and G'}.
\end{proof}

\begin{rem} \label{brauer}
Theorem \ref{construction of G and G'} is an analogue of the well-known result:
\begin{equation*} \label{exact seq of brauers}
1 \longrightarrow H^2(\Gal(\bar{\F}/\F), \F^{\times}) 
\longrightarrow \bigoplus_{v} H^2(\Gal(\bar{\F}_v/\F_v), \F_{v}^{\times})
\overset{\sum \text{inv}_v}{\longrightarrow}
\QQ / \ZZ
\longrightarrow 0.  
\end{equation*}
To be precise, this exact sequence explains how to obtain a central division algebra over $\F$ from a given central division algebra over $F$. When an $F$-group $G$ satisfies condition (\ref{cond on G and G}), this is a manner to construct an $\F$-group $\G$ and its $\F$-inner form $\G'$ with prescribed local behavior. Theorem \ref{construction of G and G'} fully extends this notion to any connected reductive group over $F$.
\end{rem}

Given an irreducible unitary supercuspidal representation, the following proposition tells us how to construct a cuspidal automorphic representation with specified local behavior at a finite set of places.
\begin{pro} (\cite[Th\'{e}or\`{e}m, in Appendice 1]{he84}) \label{he84}
Let $\F$ be a global field, $\G$ a connected reductive group over $\F$, $Z(\G)$ its center, $\omega$ a unitary character of $Z(\G) (\F) \backslash  Z(\G)(\AA)$,  $S$ a nonempty finite set of finite places of $\F$ and, for $v \in S$, $\rho _v$ an irreducible unitary supercuspidal representation of $\G(\F_v)$ with central character $\omega_v$. Then, there exists a cuspidal automorphic representation $\pi=\otimes_v \pi_v$ of $\G(\AA)$ with central character $\omega$ such that $\pi_v \s \rho_v$ for all $v \in S$.
\end{pro}
\begin{rem}
If the set $S$ is chosen so that each group $\G_v$ is unramified over $\F_v$ at all finite place $v \notin S$, then each representation $\pi_v$ is unramified for all finite place $v \notin S$ by the choice of each $f_v$ in the proof of \cite[Th\'{e}or\`{e}m, in Appendice 1]{he84} (cf. \cite[Section 5]{sh90}).
\end{rem}
\subsection{Local and Global Compatibility in Restriction} \label{local and global compatibility}
In this section, we establish the local and global compatibility in the restriction of representations from a group to its subgroup sharing the same derived group. Let $F$ denote a $p$-adic field of characteristic $0$ with the ring of integers $\OO_F$.

\begin{pro} \label{restriction of unramified representation}
Let $G$ and $\widetilde{G}$ be unramified groups over $F$ such that
\[
G_{\der} = \widetilde{G}_{\der} \subseteq G \subseteq \widetilde{G}. 
\]
Given an unramified representation $\widetilde{\tau}$ of $\widetilde{G}(F)$, its restriction $\widetilde{\tau} |_{G(F)}$ to $G(F)$ has a unique unramified constituent with respect to $G(\OO_F)$.
\end{pro}
\begin{proof}
Fix a Borel subgroup $\widetilde{B}=\widetilde{T} U$ of $\widetilde{G}$. Then we have a Borel subgroup $B=\widetilde{B} \cap G = (\widetilde{T} \cap G) U$ of $G$. Write $T=\widetilde{T} \cap G$. From \cite[Proposition 2.6]{ca80}, we have a $\widetilde{G}(F)$-embedding of $\widetilde{\tau}$ into an unramified principal series $\ii_{\widetilde{G}, \widetilde{B}} \tlchi$, where $\tlchi$ is an unramified character of $\widetilde{T}(F)$. Consider the restriction  
$(\ii_{\widetilde{G}, \widetilde{B}} \tlchi) | _{G(F)}$. We note that $\widetilde{G}(F) = \widetilde{T}(F) G(F)$, and $f(\tilde{t}g)= \tlchi(\tilde{t})f(g)$ 
for $\tilde{t} \in \widetilde{T}(F)$ and $g \in G(F)$. It follows that if $f | _{G(F)} = 0$ for $f \in \ii_{\widetilde{G}, \widetilde{B}} \tlchi$, then $f = 0$. By sending $f \mapsto f | _{G(F)}$, we thus have a $G(F)$-equivariant embedding
\begin{equation*} \label{equ restriction of unramified rep}
\ii_{\widetilde{G}, \widetilde{B}} \tlchi
\hookrightarrow
\ii_{G, B} \chi.
\end{equation*}
Here $\chi$ is the restriction of $\tlchi$ to $T(F)$. 
Since $\chi$ is an unramified character of $T(F)$, we note that 
$\ii_{G, B} \chi$ has 
a unique non-trivial spherical vector with respect to $G(\OO_F)$ up to scalar. Hence, there is exactly one irreducible constituent of $\ii_{G, B} \chi$ which contains the unique non-trivial spherical vector. This completes the proof.
\end{proof}

Let $\F$ denote a number field, and let $\AA$ denote the adele ring of $\F$. Let $\G$ and $\wG$ be connected reductive groups over $\F$ such that
\[
\wG_{\der} = \G_{\der} \subseteq \G \subseteq \wG.
\] 
Let $\widetilde{\pi} = \otimes_v \widetilde{\pi}_v$ be an irreducible admissible representation of $\wG(\AA)$.
Proposition \ref{restriction of unramified representation} implies that $\widetilde{\pi}_v |_{\G(\F_v)}$ has 
a unique unramified constituent (denoted by $\pi^0_v$) with respect to $\G(\OO_v)$ for almost all places $v$, where $\OO_v$ is the ring of integers of $\F_v$. 
The following proposition states the local and global compatibility in the restriction which is an analogue of \cite[Lemma 1]{lan79}.
\begin{pro} \label{restriction of global rep}
Every irreducible constituent of $\widetilde{\pi}|_{\G(\AA)}$ is of the form 
$\pi = \otimes_v \pi_v$, where $\pi_v$ is an irreducible constituent of 
$\widetilde{\pi}_v |_{\G(\F_v)}$ and $\pi_v \s \pi^0_v$ for almost all $v$.
\end{pro}
\begin{proof}
We follow the proof of \cite[Lemma 1]{lan79}. It is clear that any representation of the above form $\pi = \otimes_v \pi_v$ is an irreducible constituent of $\widetilde{\pi}|_{\G(\AA)}$. Conversely, let the constituent 
$\pi$ act on $V/U$ with $0 \subseteq U \subseteq V \subseteq X=\otimes_v X_v$. To construct the tensor product, we choose a finite set $T_0$ of places and a non-zero spherical vector $x_0$ with respect to $\G(\OO_v)$ for each $v \notin T_0$. 
We can find a finite set $T$ of places and a vector $x_T \in X_T:=\otimes_{v \in T} X_v$ such that $T$ contains $T_0$ and $x = x_T \otimes(\otimes_{v \notin T}x_v^0)$ lies in $V$ but not in $U$.
Let $V_T$ be the smallest subspace of $X_T$ containing $x_T$ and invariant under 
$\G_T:= \Pi _{v \in T} \G(\F_v)$. There is a surjective map
\[
V_T \otimes (\otimes_{v \notin T} V_v) \longrightarrow V/U.
\]
If $v_0 \notin T$, then the kernel contains $V_T \otimes U_{v_0} \otimes(\otimes_{v \notin T \cup \lbrace v_0 \rbrace} V_v)$.
We obtain a surjection $V_T \otimes (\otimes_{v \in T} V_v / U_v) \rightarrow V/U$ with a kernel of the form $U_T \otimes (\otimes_{v \in T} V_v / U_v)$, where $U_T$ lies in $V_T$. We note from \cite[Lemma 2.1]{tad92} that $\widetilde{\pi}_v |_{\G(\F_v)}$ is a finite direct sum of irreducible constituents of $\G(\F_v)$. So, the representation of $\G_T$ on $V_T / U_T$ is irreducible and is of the form 
$\otimes_{v \in T} \pi_v$, where $\pi_v$ is an irreducible constituent of $\widetilde{\pi}_v |_{\G(\F_v)}$. Thus, the constituent $\pi$ is of the form $\otimes_v \pi_v$ such that $\pi_v \s \pi^0_v$ for $v \notin T$.
\end{proof}

\begin{rem} \label{restriction of cuspidal rep}
Let $\widetilde{\pi}$ be a cuspidal automorphic representation
of $\wG(\AA)$. Since $\G(\AA)$ is a subgroup of $\wG(\AA)$ sharing the same derived group,
the restriction $\mathcal{H}(\widetilde{\pi})|_{\G(\AA)}$ of the space $\mathcal{H}(\widetilde{\pi})$ 
from $\wG(\AA)$ to $\G(\AA)$ is still a non-zero subspace of cusp forms on $\G(\AA)$.
Due to the result \cite[Theorem 2.1 p.113]{ckm03} of Gelfand and Piatetski-Shaphiro, we have the decomposition 
\[
\widetilde{\pi}|_{\G(\AA)} = \oplus m_{\pi} \pi.
\]
Here $\pi$ runs through all irreducible constituents of the restriction $\widetilde{\pi}|_{\G(\AA)}$ which are cuspidal automorphic representations of $\G(\AA)$, and $m_{\pi}$ is the multiplicity of $\pi$.
\end{rem}

\begin{rem} \label{local rest}
Let $\pi = \otimes \pi_v$ be a cuspidal automorphic representation  of $\G(\AA)$. Theorem 4.13 in \cite{hs11} verifies that there exists a cuspidal automorphic representation $\widetilde{\pi} = \otimes \widetilde{\pi}_v$ of $\wG(\AA)$ such that $\pi$ is an irreducible constituent of $\widetilde{\pi}|_{\G(\AA)}$. From Proposition \ref{restriction of global rep}, we note that $\pi_v$ is an irreducible constituent of $\widetilde{\pi}_v|_{\G(\F_v)}$ for all $v$.
\end{rem}
\section{Transfer for Unitary Supercuspidal Representations under the Local JL-Type Correspondence}  \label{main thms}
Using a local to global argument in Section \ref{local to global}, we prove that Plancherel measures attached to unitary supercuspidal representations are preserved under the local JL-type correspondence, assuming a working hypothesis that Plancherel measures are invariant on a certain finite set.
 Throughout Section \ref{main thms}, $F$ denotes a $p$-adic field of characteristic $0$, and $M$ denotes an $F$-Levi subgroup of a connected reductive $F$-split group $G$ such that
\begin{equation} \label{cond on M in body}
\Pi_{i=1}^{r} SL_{n_i} \subseteq M \subseteq \Pi_{i=1}^{r} GL_{n_i} 
\end{equation}
for positive integers $r$ and $n_i$. Let $G'$ be an $F$-inner form of $G$, and let $M'$ be an $F$-Levi subgroup of $G'$ that is an $F$-inner form of $M$. Then, $M'$ satisfies the following property
\begin{equation*} \label{cond on M' in body}
\Pi_{i=1}^{r} SL_{m_i}(D_{d_i}) \subseteq M'(F) \subseteq \Pi_{i=1}^{r} GL_{m_i}(D_{d_i}).
\end{equation*}
Here $D_{d_i}$ denotes a central division algebra of dimension $d_i^2$ over $F$ where $n_i = m_i d_i$. 
Write $\tM(F) = \Pi_{i=1}^{r} GL_{n_i}(F)$ and $\tM'(F) = \Pi_{i=1}^{r} GL_{m_i}(D_{d_i})$. 

We denote by $\theta$ and $\theta'$ the subsets of $\Delta$ and $\Delta'$ such that $M=M_{\theta}$ and $M=M_{\theta'}$, respectively. We fix representatives $w \in G(F)$ and $w' \in G'(F)$ of $\tilde{w} \in W_G$ and $\tilde{w} \in W_{G'}$ such that $\tilde{w} = \varphi(\tilde{w}')$, $\tilde{w}(\theta) \subseteq \Delta$ and $\tilde{w}'(\theta') \subseteq \Delta'$ as stated in Section \ref{inner forms}.
\subsection{Statement of Theorem}
In this section, we state the main result and its contributions. Let $\sigma \in \e(M(F))$ and $\sigma' \in \e(M'(F))$ be under the local JL-type correspondence. Since both $\sigma$ and $\sigma'$ are supercuspidal, by Remark \ref{remark for sc to sc} and Definition \ref{def of LJL-type}, we have $\tsigma \in \e(\tM(F))$ and $\tsigma' \in \e(\tM'(F))$ such that
\begin{itemize}
  \item[(a)] $\sigma$ and $\sigma'$ are isomorphic to irreducible constituents of the restrictions $\tsigma |_{M(F)}$ and $\tsigma' |_{M'(F)}$, respectively,
  \item[(b)] $\mathbf{C}(\tlsigma) = \tlsigma'$.
\end{itemize}
We recall the sets $\Pi_{\tsigma}(M(F))$ and $\Pi_{\tsigma'}(M'(F))$ of equivalence classes of all irreducible constituents of $\tsigma|_{M(F)}$ and $\tsigma'|_{M'(F)}$, respectively. 

\begin{wh}  \label{whyp}
Let $\sigma'_1$ and $\sigma'_2$ be given in $\Pi_{\tsigma'}(M'(F))$. Then, we have 
\[
\mu_{\bM'}(\nu', \sigma'_1, w') = \mu_{\bM'}(\nu', \sigma'_2, w')
\]
for any $\nu' \in \ma^{*}_{\bM', \CC}$.
\end{wh}
\begin{rem}
Since $\Pi_{\tsigma'}(M'(F))$ can be considered as an $L$-packet on $M'$ (cf. Remark \ref{rem of l-packet}), this hypothesis is related to the $L$-packet invariance of the Plancherel measure. When $G'$ is an $F$-inner form of $SL_n$ and $M'$ is any $F$-Levi subgroup of $G'$, Working Hypothesis \ref{whyp} is a consequence of Proposition \ref{Compatibility}.
\end{rem}
The following states our main result.
\begin{thm} \label{relation bw JL-type and PM}
Let $\sigma \in \e(M(F))$ and $\sigma' \in \e(M'(F))$ be under the local JL-type correspondence. Assume that Working Hypothesis \ref{whyp} is valid. Then, we have
\[
\mu_{\bM}(\nu, \sigma, w) = \mu_{\bM'}(\nu, {\sigma}', w')
\]
for any $\nu \in \ma^{*}_{\bM, \CC} \s \ma^{*}_{\bM', \CC}$.
\end{thm}
\begin{rem}
If the central character of $\tsigma'$ is unramified and the set $\Pi_{\tsigma'}(M'(F))$ is associated to a tempered and tame regular semi-simple elliptic $L$-parameter, the assumption is no longer needed due to \cite{choiy2}. 
\end{rem}
We note that any two members $\tau \in \Pi_{\tsigma}(M(F))$ and $\tau' \in \Pi_{\tsigma'}(M'(F))$ are under the JL-type correspondence (see Remark \ref{an rem for JL-type}). Fix $\sigma' \in \e(M'(F))$ in Theorem \ref{relation bw JL-type and PM}. By varying $\sigma$ over $\Pi_{\tsigma}(M(F))$, we have the following.
\begin{pro}  \label{L-packet invariance of PM}
Let $\tau_1$ and $\tau_2$ be given in $\Pi_{\tsigma}(M(F))$. Assume that Working Hypothesis \ref{whyp} is valid. Then, we have
\[
\mu_M(\nu, \tau_1, w) = \mu_M(\nu, \tau_2, w)
\]
for all $\nu \in \mathfrak{a}^{*}_{M, \CC}$.
\end{pro}
\begin{rem} \label{invariant rmk}
Since $\Pi_{\tsigma}(M(F))$ can be considered as an $L$-packet on $M$ (cf. Remark \ref{rem of l-packet}), Proposition \ref{L-packet invariance of PM} supports the conjecture that Plancherel measures are invariant on $L$-packets. We also refer the reader to \cite{art11, choiy2, gtsp10, gt, gtan12} for other cases.
\end{rem}
\begin{rem} \label{remark for generic L-packet}
Proposition \ref{L-packet invariance of PM} reduces the study of the Plancherel measures for tempered representations to the generic cases. To be precise, since the unitary supercuspidal representation $\tsigma$ of $\tM(F)$ is generic with respect to a generic character $\psi$ \cite[Section 2.3]{ku94}, there exists a unique generic representation $\tau_0 \in \Pi_{\tsigma}(M(F))$ with respect to $\psi$ (cf. \cite[Proposition 2.8]{tad92}). Then, the result \cite[Corollary 3.6]{sh90} of Shahidi gives an explicit formula of $\mu_M(\nu, \tau_0, w)$ in terms of local factors via the Langlands-Shahidi method. Therefore, due to Proposition \ref{L-packet invariance of PM}, the Plancherel measure $\mu_M(\nu, \tau, w)$ attached to any $\tau \in \Pi_{\tsigma}(M(F))$ has the same formula with $\mu_M(\nu, \tau_0, w)$. Moreover, Theorem \ref{relation bw JL-type and PM} admits the same formula for $\mu_{M'}(\nu, \tau', w')$ with any $\tau' \in \Pi_{\tsigma'}(M'(F))$. 
\end{rem}
\subsection{Proof of Theorem \ref{relation bw JL-type and PM}} \label{pf of thm for st}
In this section, we present the proof of Theorem \ref{relation bw JL-type and PM} by a local to global argument in Section \ref{local to global}. The outline is as follows.

($Step \, \mathit{1}$) Construct global data from given local data as explained in Section \ref{const from l to g}.

($Step \, \mathit{2}$) Find two cuspidal automorphic representations as described in Proposition \ref{existence pro}. 

($Step \, \mathit{3}$) Use Langlands' functional equation on Eisenstein series.

This is a more general version of the method of Mui{\'c} and Savin in \cite{ms00}. Along with their Siegel Levi subgroups ($\s GL_n$) of the groups $SO_{2n}$ and $Sp_{2n}$, our construction in ($Step \, \mathit{1}$) treats any $F$-Levi subgroup of connected reductive groups over a $p$-adic field of characteristic $0$. Further, ($Step \, \mathit{2}$) can be applied to any connected reductive group $M$ satisfying condition (\ref{cond on M in body}).

The rest of this section is devoted to the proof of Theorem \ref{relation bw JL-type and PM}. Due to Remarks \ref{useful rmks on pm} and \ref{remark for unitary and max}, it suffices to consider the case when $G$ is semisimple (so is $G'$) and $M$ is maximal (so is $M'$). It then turns out that $\ma^{*}_{M, \CC} \s \ma^{*}_{M', \CC} \s \CC$. Thus, we shall show that: for all $s \in \CC$,
\[
\mu_{M}(s, \sigma, w) = \mu_{M'}(s, {\sigma}', w').
\]
We start with the following Lemma which is immediately a consequence of Theorem \ref{construction of G and G'}.
\begin{lm} \label{const lemma}
Let $F$, $G$, $G'$, $M$, $M'$, $\tM$ and $\tM'$ be as above. Then, there exist a number field $\F$, a non-empty finite set $S$ of finite places of $\F$, a connected reductive $\F$-split group $\G$ and its $\F$-inner form $\G'$, their $\F$-Levi subgroups $\M$ and $\M'$ ( $\F$-inner form of $\M$ ), a connected reductive $\F$-split group $\wM$ and its $\F$-inner form $\wM'$ such that
\begin{itemize}
  \item[(a)]  for all $v \in S$, $\mathbf{F}_v \s F$,  $\mathbf{G}_v \s G$, $\mathbf{G}'_v \s G'$, $\mathbf{M}_v \s M$, $\mathbf{M}'_v \s M'$, $\wM_v \s \tM$ and $\wM'_v \s \tM'$ over $\mathbf{F}_v$, 
  \item[(b)]  for all $v \notin S$, $\mathbf{G}_v \s \mathbf{G}'_v$, $\mathbf{M}_v \s \mathbf{M}'_v$ and  $\mathbf{\tM}_v \s \mathbf{\tM}'_v$ over $\mathbf{F}_v$,
  \item[(c)]  $\M_{\der} = \wM_{\der} \subseteq \M \subseteq \wM$ and $\M'_{\der} = \wM'_{\der} \subseteq \M' \subseteq \wM'$.
\end{itemize}
\end{lm}
\noindent Note that $\wM(\F)$ and  $\wM'(\F)$ are of the forms $\prod_{i=1}^{r} \mathbf{GL}_{n_i}(\F)$ and $\prod_{i=1}^{r} \mathbf{GL}_{m_i}(\D_{d_i})$, respectively. Here $\D_{i}$ denotes a central division algebra of dimension $d_{i}^{2}$ over $\F$ and $n_i = m_i d_i$. For all $v \in S$, it turns out that $\mathbf{G}'_v$, $\mathbf{M}'_v$ and $\mathbf{\tM}'_v$ are \textit{non quasi-split} $F_v$-inner forms of $\mathbf{G}_v$, $\mathbf{M}_v$ and $\wM_v$ respectively. Also, $\mathbf{G}_v$, $\mathbf{G}'_v$, $\mathbf{M}_v$ and $\mathbf{M}'_v$ are all \textit{quasi-split} over $\mathbf{F}_v$ unless $v \in S$.

From now on, we fix a number field $\F$ and a finite set $S$ of finite places of $\F$ as defined in Lemma \ref{const lemma}, so that $\mathbf{F}_v \s F$ for all $v \in S$ and $S$ consists of all non-split places. Next, we find two following cuspidal automorphic representations $\pi = \otimes_v \pi_v$ of $\M(\AA)$ and $\pi' = \otimes_v \pi'_v$ of $\M'(\AA)$.
\begin{pro} \label{existence pro}
Let $\sigma \in \e(M(F))$ and $\sigma' \in \e(M'(F))$ be under the local JL-type correspondence. Then, there exist a finite set $V$ of places of $\F$ containing $S$ and all archimedean places, and two cuspidal automorphic representations $\pi = \otimes_v \pi_v$ of $\M(\AA)$ and $\pi' = \otimes_v \pi'_v$ of $\M'(\AA)$ such that
\begin{itemize}
  \item[(a)] for all $v \in S$, $\pi_v \s \sigma$ and $\pi'_v \in \Pi_{\tsigma'}(M'(F))$,
  \item[(b)] for all  $v \in V - S$, $\pi_v $ and $\pi'_v$ are irreducible constituents of the restriction of an irreducible representation of $\wM(\F_v)$ to $\M(\F_v)$,
  \item[(c)] for all $v \notin V$, $\pi_v$ and $\pi'_v$ are isomorphic and unramified with respect to $\M(\OO_v)$.
\end{itemize}
\end{pro}
\begin{proof} [Proof of Proposition \ref{existence pro}]
From Proposition \ref{he84}, we construct a cuspidal automorphic representation $\pi = \otimes \pi_v$ of $\M(\AA)$ such that $\pi_v \s \sigma$
for all $v \in S$. By Remark \ref{local rest}, we also have a cuspidal automorphic representation $\widetilde{\pi} = \otimes \widetilde{\pi}_v$ of $\wM(\AA)$ such that $\pi$ is an irreducible constituent of $\widetilde{\pi}|_{\M(\AA)}$ and $\pi_v$ is an irreducible constituent of $\widetilde{\pi}_v|_{\M(\F_v)}$ for all $v$. 

We claim that $\widetilde{\pi}$ is in the image of the map $\Phi$ defined in Theorem \ref{global JL}. To see this, we need to show that $\widetilde{\pi}$ is $D$-compatible. For all $v \in S$, since $\pi_v$ is in $\e(\M(\F_v))$, $\widetilde{\pi}_v$ is also in $\e(\wM(\F_v))$ by Remark \ref{remark for sc to sc}. It follows from remark \ref{remark for extended d-compatible} that $\widetilde{\pi}_v$ is $d_v$-compatible for all $v \in S$. Hence, from Definition \ref{def of D-compatible}, $\widetilde{\pi}$ is $D$-compatible. We note from Proposition \ref{JL for cuspidal} that $\widetilde{\pi}'$ is cuspidal since $\widetilde{\pi}$ is cuspidal. Therefore, we have a unique cuspidal automorphic representation $\widetilde{\pi}'$ of $\widetilde{\M}'(\AA)$ such that
$|\mathbf{LJ}|_v(\pi_v) = \pi'_v$ for all $v$.
Since both $\widetilde{\pi}_v$ and $\widetilde{\pi'}_v$ are supercuspidal for $v \in S$ and $S$ consists of non-split places, we note that $|\mathbf{LJ}|_v(\widetilde{\pi}_v) = \mathbf{C}(\widetilde{\pi}_v)$ for all $v \in S$ by Remark \ref{rem of |JL|} and $\widetilde{\pi}'_v \s \widetilde{\pi}_v$ for all $v \notin S$. 

Now we consider the restriction $\widetilde{\pi}'|_{\M'(\AA)}$ of $\widetilde{\pi}'$ to $\M'(\AA)$. From Remark \ref{restriction of cuspidal rep}, there exists a cuspidal automorphic representation $\pi'= \otimes_v \pi'_v$ of $\M'(\AA)$ such that $\pi'_v$ is an irreducible constituent of $\widetilde{\pi}'_v|_{\M'(\F_v)}$ for all $v$.

The assertions (a), (b) and (c) are verified as follows. For all $v \in S$, since $\sigma$ is an irreducible constituent of both $\tsigma|_{\M(\F_v)}$ and $\widetilde{\pi}_v|_{\M(\F_v)}$, it follows from Remark \ref{remark for lifting} that  $\widetilde{\pi}_v \s \tsigma \otimes (\eta_v \circ \det)$ for some quasi-character $\eta_v$ of $\F_v^{\times}$ (in fact, $\eta_v$ is unitary). 
Remark \ref{torsion by character} yields that 
$\widetilde{\pi}'_v \s \tsigma' \otimes (\eta_v \circ \Nrd)$. So, we have from Remark \ref{remark for lifting} that $\pi'_v$ lies in $\Pi_{\tsigma'}(M'(F))$ for all $v \in S$.
Since $\widetilde{\pi}'_v \s \widetilde{\pi}_v$ for all $v \notin S$, Proposition \ref{restriction of global rep} allows us to have a finite set $V$ containing $S$ such that $\pi_v$ and $\pi'_v$ are isomorphic and unramified for all $v \notin V$ with respect to $\M(\OO_v)$. Also, for all $v \in V - S$, $\pi_v$ and $\pi'_v$ are irreducible constituents of the restriction of $\widetilde{\pi}'_v \s \widetilde{\pi}_v$ of $\wM(\F_v)$ to $\M(\F_v)$. This completes the proof of Proposition \ref{existence pro}
\end{proof}

We consider the standard global intertwining operator $M(s, \pi) := \otimes_v A(s, \pi_v, w)$ from the global induced representation $I(s, \pi)$ to $I(-s, w(\pi))$. We refer to \cite[Sections 4 and 7]{ckm03}. For $f \in I(s, \pi)$, we note that
\[
f \in \otimes_{v \in V} I(s, \pi_v) \otimes (\otimes_{v \notin V} f^0_v), 
\]
where $f^0_v$ is a unique spherical vector such that $f^0_v | _{\M(\OO_v)} = 1$. Let $M(s, \pi')$, $I(s, \pi')$ and $f'^{0}_v$ have the corresponding meaning for the cuspidal automorphic representation $\pi'$ of $\M'(\AA)$. Due to Proposition \ref{existence pro}, we identify $f^0_v$ and $f'^{0}_v$. Given $f = \otimes_v f_v \in I(s, \pi)$ and $f' = \otimes_v f'_v \in I(s, \pi')$, we get
\[
M(s, \pi) f = \otimes_{v \in V} A(s, \pi_v, w) f_v \otimes (\otimes_{v \notin V} A(s, \pi_v, w) f^0_v),
\] 
\[
M(s, \pi') f' = \otimes_{v \in V} A(s, \pi'_v, w') f'_v \otimes (\otimes_{v \notin V} A(s, \pi'_v, w') f^0_v).
\] 
So, we have the following functional equations by Eisenstein series 
\[
M(s, \pi)M(-s, w(\pi))= id,
\]
\[
M(s, \pi')M(-s, w(\pi'))= id.
\] 
It then follows that
\begin{equation} \label{identity of intertwining operators of M}
\prod_{v \in V}  \mu_{\M_v}(s, \pi_v, w) \gamma _{\tilde{w}}(\G_v|\M_v)^{-2}  \prod_{v \notin V} c_v(s, \pi_v)c_v(-s, w(\pi_v)) = 1,
\end{equation}
\begin{equation}\label{identity of intertwining operators of M'}
\prod_{v \in V}  \mu_{\M'_v}(s, \pi'_v, w') \gamma _{\tilde{w}}(\G'_v|\M'_v)^{-2} \prod_{v \notin V} c_v(s, \pi'_v)c_v(-s, w(\pi'_v)) = 1.
\end{equation}
Here $c_v( \cdot , \cdot )$ is a quotient of the local Langlands $L$-functions for unramified representations (see \cite[(2.7) p.554]{sh88}). Since $\pi_v \s \pi'_v$ for all $v \notin V$, we have 
\[
c_v(s, \pi_v)c_v(-s, w(\pi_v)) = c_v(s, \pi'_v)c_v(-s, w(\pi'_v)).
\]
For all $v \in V-S$, since $\pi_v $ and $\pi'_v$ are irreducible constituents of the restriction of an irreducible representation (denoted by $\widetilde{\tau}_v$) of $\wM(\F_v)$ to $\M(\F_v)$, both $\mu_{\M_v}(s, \pi_v, w)$ and $\mu_{\M'_v}(s, \pi'_v, w')$ can be expressed in terms of the Artin $L$-function and root number attached to the $L$-parameter of $\widetilde{\tau}_v$ (see \cite[Proposition 2.6]{ca80} and \cite[Theorem 3.1]{keysh} for non-archimedean places, \cite[Section 3]{arthur89f} for archimedean places). It then follows that $\mu_{\M_v}(s, \pi_v, w) = \mu_{\M'_v}(s, \pi'_v, w')$.
Further, we note that $\gamma _{\tilde{w}}(\G_v|\M_v)=\gamma _{\tilde{w}}(\G'_v|\M'_v)$ by \cite[p.89]{ac89}. So, we have
\[
\prod_{v \in S}  \mu_{\M_v}(s, \pi_v, w) = \prod_{v \in S}  \mu_{\M'_v}(s, \pi'_v, w').
\]
For all $v \in S$, we note that: $\M_v \s M$ and $\M'_v \s M'$ over $\F_v$ by Lemma \ref{const lemma}; $\pi_v \s \sigma$ by Proposition \ref{existence pro}; $\mu_{\M'_v}(s, \pi'_v, w') = \mu_{\M'_v}(s, \sigma', w')$ by Working Hypothesis \ref{whyp}. 
Hence, we deduce from equations (\ref{identity of intertwining operators of M}) and (\ref{identity of intertwining operators of M'}) that 
\begin{equation} \label{equality equality of two pm}
\mu_{M}(s, \sigma, w)^m = \mu_{M'}(s, \sigma', w')^m. 
\end{equation}
Here $m$ denotes the cardinality of $S$.
Since Plancherel measures are holomorphic and non-negative along the unitary axis $Re(s)=0$, we therefore have that
\[
\mu_{M} (s, \sigma, w) = \mu_{M'} (s, \sigma', w')
\]
for all $s \in \CC$. This completes the proof of Theorem \ref{relation bw JL-type and PM}.
\section{Applications} \label{ex and appl}
In this section, we present some applications of Theorem \ref{relation bw JL-type and PM}. We continue with the notation in Section \ref{main thms}. Throughout Section \ref{ex and appl}, we assume that $M$ and $M'$ are maximal.
For $\sigma \in \Irr(M(F))$ and $\sigma' \in \Irr(M'(F))$, we define $W(\sigma) := \lbrace  w \in W_M : \; ^{w}\sigma \s \sigma \rbrace
~ \text{and} ~ 
W(\sigma') := \lbrace  w' \in W_{M'} : \; ^{w'}\sigma' \s \sigma' \rbrace.
$
\begin{pro} \label{pro of transfer of red}
Let $\sigma \in \e(M(F))$ and $\sigma' \in \e(M'(F))$ be under the local JL-type correspondence. Suppose that $|W(\sigma)| = |W(\sigma')| = 2$. Let $\nu_0 \in \RR$ be given. Then, under Working Hypothesis \ref{whyp}, $\ii_{G,M}(\nu, \sigma)$ is reducible at $\nu=\nu_0$ if and only if $\ii_{G',M'}(\nu, \sigma')$ is reducible at $\nu=\nu_0$. Moreover, $\nu_0$ is either $0$ or $ \pm x_0 $ for some positive real number $x_0$.
\end{pro}
\begin{proof}
If $\mu_{M}(0, \sigma, w) \neq 0$, we have $\mu_{M'}(0, \sigma', w') \neq 0$ from Theorem \ref{relation bw JL-type and PM}. Since $|W(\sigma)| = |W(\sigma')| = 2$, we note from \cite[Corollary 5.4.2.3]{sil79} that both $\ii_{G,M}(0, \sigma)$ and $\ii_{G',M'}(0, \sigma')$ are reducible. Also, it follows from \cite[Lemma 1.3]{sil80} that both $\mu_{M}(\nu, \sigma, w)$ and $\mu_{M'}(\nu, \sigma', w')$ are holomorphic for all $\nu \in \RR - \{ 0 \}$. Hence, by \cite[Lemma 5.4.2.4]{sil79}, both $\ii_{G,M}(\nu, \sigma)$ and $\ii_{G',M'}(\nu, \sigma')$ are irreducible for all $\nu \in \RR - \{0 \}$.

If $\mu_{M}(0, \sigma, w) = 0$, Theorem \ref{relation bw JL-type and PM} implies that $\mu_{M'}(0, \sigma', w') = 0$. So, by \cite[Corollary 5.4.2.3]{sil79}, both $\ii_{G,M}(0, \sigma)$ and $\ii_{G',M'}(0, \sigma')$ are irreducible. It follows from \cite[Lemma 1.2]{sil80} that, for $\nu \in \RR$, there exists a unique $x_0 > 0$ such that $\mu_{M}(\nu, \sigma, w)$ has a (simple) pole at $\nu = \pm x_0$. We note from \cite[Lemma 5.4.2.4]{sil79} that $\ii_{G,M}(\nu, \sigma)$ with $\nu \in \RR$ is reducible only at $\nu = \pm x_0$. By Theorem \ref{relation bw JL-type and PM}, the same is true for $\ii_{G',M'}(\nu, \sigma')$. Therefore, both are irreducible for all $\nu \in \RR - \{ \pm x_0 \}$.
\end{proof}
\begin{rem}
If $|W(\sigma)| = |W(\sigma')| = 1$, both $\ii_{G,M}(\nu, \sigma)$ and $\ii_{G',M'}(\nu, \sigma')$ are irreducible for any $\nu \in \RR$ due to \cite[Lemma 1.3]{sil80} and \cite[Lemma 5.4.2.4]{sil79}. 
\end{rem}
For $\sigma \in \Irr_u(M(F))$ and $\nu \in \RR$, we say that $\ii_{G,M}(\nu, \sigma)$ is \textit{in the complementary series} if it is unitarizable. The following is immediately a consequence of Proposition \ref{pro of transfer of red}.
\begin{cor} \label{rem for comp}
Assume as in Proposition \ref{pro of transfer of red}. If $\ii_{G,M}(0, \sigma)$ is irreducible, then $\ii_{G,M}(\nu, \sigma)$ is in the complimentary series if and only if $\ii_{G',M'}(\nu, \sigma')$ is in the complimentary series if and only if $|\nu| < x_0$. 
If $\ii_{G,M}(0, \sigma)$ is reducible, then neither $\ii_{G,M}(\nu, \sigma)$ nor $\ii_{G',M'}(\nu, \sigma')$ is in the complimentary series for $\nu > 0$.
\end{cor}
\begin{rem}
Proposition \ref{pro of transfer of red} and Corollary \ref{rem for comp} imply that the reducibility of $\ii_{G,M}(\nu, \sigma)$ and the edges of the complimentary series are transferred to those of $\ii_{G',M'}(\nu, \sigma')$ for $\nu \in \RR$.
\end{rem}
More explicit values of $x_0$ are made in the following corollary which is a consequence of Remark \ref{remark for generic L-packet}, Proposition \ref{pro of transfer of red}, Corollary \ref{rem for comp} and \cite[Theorem 8.1]{sh90}.
\begin{cor} \label{cor for comp}
Assume as in Proposition \ref{pro of transfer of red}. Suppose $\ii_{G,M}(0, \sigma)$ is irreducible. Choose a unique $i$, $i=1$ or $2$, such that $P_{\sigma, i}(1)=0$ (see \cite[Corollary 7.6]{sh90} for details). Then,
\begin{itemize}
  \item[(a)] the real number $x_0$ in Proposition \ref{pro of transfer of red} is either $\frac{1}{2}$  or  $ 1 $,
  \item[(b)] for $0 < |\nu| < 1/i$, both $\ii_{G,M}(\nu, \sigma)$ and $\ii_{G',M'}(\nu, \sigma')$ are in the complementary series,
  \item[(c)] for $|\nu| > 1/i$, neither $\ii_{G,M}(\nu, \sigma)$ nor $\ii_{G',M'}(\nu, \sigma')$ is in the complementary series.
\end{itemize}
\end{cor}
In what follows, we prove that both the reducibility of $\ii_{G,M}(\nu, \sigma)$ and the edges of complementary series are invariant on the set $\Pi_{\tsigma}(M(F))$ when $M$ is maximal.

\begin{pro} \label{pro of L-packet invariance of red}
Let $\sigma_1$, $\sigma_2 \in \Pi_{\tsigma}(M(F))$ be given. Suppose that $|W(\sigma_1)| = |W(\sigma_2)| = 2$. Let $\nu_0 \in \RR$ be given. Then, under Working Hypothesis \ref{whyp}, $\ii_{G,M}(\nu, \sigma_1)$ is reducible at $\nu=\nu_0$ if and only if $\ii_{G,M}(\nu, \sigma_2)$ is reducible at $\nu=\nu_0$. Moreover, $\nu_0$ is either $0$ or $ \pm x_0 $ for some positive real number $x_0$.
\end{pro}
\begin{proof}
Fix $\sigma' \in \Pi_{\tsigma'}(M'(F))$ in Proposition \ref{pro of transfer of red}. By varying $\sigma$ over $\Pi_{\tsigma}(M(F))$, we have the proposition.
\end{proof}
In a similar way to Corollary \ref{rem for comp}, we have the following.
\begin{cor} \label{rem for comp L-packet inv}
Assume as in Proposition \ref{pro of L-packet invariance of red}. If $\ii_{G,M}(0, \sigma_1)$ is irreducible, then $\ii_{G,M}(\nu, \sigma_1)$ is in the complimentary series if and only if $\ii_{G,M}(\nu, \sigma_2)$ is in the complimentary series if and only if $|\nu| < \nu_0$. If $\ii_{G,M}(0, \sigma_1)$ is reducible, then neither $\ii_{G,M}(\nu, \sigma_1)$ nor $\ii_{G,M}(\nu, \sigma_2)$ is in the complimentary series for $\nu > 0$.
\end{cor}
\section{A Generalization} \label{Further Work}
Let $F$ be a $p$-adic field of characteristic $0$, and let $M$ be an $F$-Levi subgroup of a connected reductive $F$-group $G$. Let $G'$ be an $F$-inner form of $G$, and let $M'$ be an $F$-Levi subgroup of $G'$ that is an $F$-inner form of $M$. In this section, we extend Theorem \ref{relation bw JL-type and PM} to the case that unitary supercuspidal representations of $M(F)$ and $M'(F)$ have the same $L$-parameter under the following assumption. 
\begin{asr} \label{existence hyp}
There exist a finite set $V$ of places of $\F$ containing $S$ and all archimedean places, and two cuspidal automorphic representations $\pi = \otimes_v \pi_v$ of $\M(\AA)$ and $\pi' = \otimes_v \pi'_v$ of $\M'(\AA)$ such that
\begin{itemize}
  \item[(a)] for all $v \in S$, $\pi_v \s \sigma$ and $\pi'_v$ is in the $L$-packet of $\sigma'$,
  \item[(b)] for all  $v \in V - S$, $\pi_v $ and $\pi'_v$ are irreducible constituents of the restriction of an irreducible representation of $\wM(\F_v)$ to $\M(\F_v)$,
  \item[(c)] for all $v \notin V$, $\pi_v$ and $\pi'_v$ are isomorphic and unramified.
\end{itemize}
\end{asr}
\begin{rem} \label{remark existence}
This assumption has been fulfilled in Proposition \ref{existence pro} when a given Levi subgroup $M$ satisfies condition (\ref{cond on M in body}),
and generalizes the global Jacquet-Langlands correspondence for $GL_n$ to any connected reductive group.
\end{rem}

Now we state the following proposition which is a generalization of Theorem \ref{relation bw JL-type and PM}.
\begin{pro} \label{main theorem}
Let $\sigma$ and $\sigma'$ be irreducible unitary supercuspidal representations of $M(F)$ and $M'(F)$ having the same $L$-parameter. Suppose that Assumption \ref{existence hyp} is valid and that Plancherel measures are invariant on the $L$-packet of $\sigma'$. Then, we have
\[
\mu_{M}(\nu, \sigma, w) = \mu_{M'}(\nu, {\sigma}', w')
\]
for $\nu \in \ma^{*}_{\bM, \CC} \s \ma^{*}_{\bM', \CC}$.
\end{pro}
\begin{proof}
This is proved by replacing Proposition \ref{existence pro} with Hypothesis \ref{existence hyp} from the proof of Theorem \ref{relation bw JL-type and PM}.
\end{proof}
\appendix
\section{Examples} \label{examples}
We continue with the notation in Sections \ref{prelim} and \ref{main thms}.
We give a few examples of an $F$-Levi subgroup $M$ and its $F$-inner form $M'$ satisfying condition (\ref{cond on M in body}) based on the Satake classification \cite[Section 3]{sa71}.
In the following diagrams (Satake diagrams) a black vertex indicates a root in the set of simple roots of the fixed minimal $F$-Levi subgroup $M'_0$ of $G'$. So, we remove only white vertices to obtain an $F$-Levi subgroup $M'$ (see \cite[Section 2.2]{sa71} and \cite[Section I.3]{bo79}). We focus on the case that $M$ is maximal (cf. Remark \ref{remark for unitary and max}). 

\noindent (1) \textbf{$\boldsymbol{A_n}$ cases}

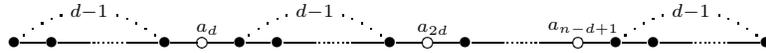
\[
\xy 
\POS (10,0) *{\bullet} ="a" ,
\POS (15,0) *{\bullet} ="b" ,
\POS (20,0) *\cir<0pt>{} ="c" ,
\POS (25,0) *\cir<0pt>{} ="d" ,
\POS (30,0) *{\bullet} ="e" ,
\POS (35,0) *\cir<2pt>{} ="f",

\POS (40,0) *{\bullet} ="a1" ,
\POS (45,0) *{\bullet} ="b1" ,
\POS (50,0) *\cir<0pt>{} ="c1" ,
\POS (55,0) *\cir<0pt>{} ="d1" ,
\POS (60,0) *{\bullet} ="e1" ,
\POS (65,0) *\cir<2pt>{} ="f1",

\POS (70,0) *{\bullet} ="a2" ,
\POS (75,0) *\cir<0pt>{} ="b2" ,
\POS (80,0) *\cir<0pt>{} ="c2" ,
\POS (85,0) *\cir<2pt>{} ="d2" ,

\POS (90,0) *{\bullet} ="a3" ,
\POS (95,0) *{\bullet} ="b3" ,
\POS (100,0) *\cir<0pt>{} ="c3" ,
\POS (105,0) *\cir<0pt>{} ="d3" ,
\POS (110,0) *{\bullet} ="e3" ,

\POS "a" \ar@{-}^<<<{}_<<{}  "b",
\POS "b" \ar@{-}^<<<{}_<<{}  "c",
\POS "c" \ar@{.}^<<<{}_<<{}  "d",
\POS "d" \ar@{-}^<<<{}_<<{}  "e",
\POS "e" \ar@{-}^<<<{}_<<{}  "f",

\POS "f" \ar@{-}^<{a_d}_<<<{} "a1",

\POS "a1" \ar@{-}^<<<{}_<<{}  "b1",
\POS "b1" \ar@{-}^<<<{}_<<{}  "c1",
\POS "c1" \ar@{.}^<<<{}_<<{}  "d1",
\POS "d1" \ar@{-}^<<<{}_<<{}  "e1",
\POS "e1" \ar@{-}^<<<{}_<<{}  "f1",

\POS "f1" \ar@{-}^<{a_{2d}}_<<<{} "a2",

\POS "a2" \ar@{-}^<<<{}_<<{}  "b2",
\POS "b2" \ar@{.}^<<<{}_<<{}  "c2",
\POS "c2" \ar@{-}^<<<{}_<<{}  "d2",

\POS "d2" \ar@{-}^<{a_{n-d+1}}_<<<{} "a3",

\POS "a3" \ar@{-}^<<<{}_<<{}  "b3",
\POS "b3" \ar@{-}^<<<{}_<<{}  "c3",
\POS "c3" \ar@{.}^<<<{}_<<{}  "d3",
\POS "d3" \ar@{-}^<<<{}_<<{}  "e3",

\POS "a" \ar@{.}@/^1pc/|{d-1}  "e",
\POS "a1" \ar@{.}@/^1pc/|{d-1}  "e1",
\POS "a3" \ar@{.}@/^1pc/|{d-1}  "e3",

\endxy
\]

Set $\theta = \Delta - \lbrace \alpha_{j} \rbrace$, where $\alpha_{j} = e_j-e_{j+1}$ for $j = d, 2d, \cdots, md$. Note that $md = n-d+1$.

\begin{itemize}
\item[(a)] Let $G = GL_{n+1}$. Note that $G'(F) = GL_{m+1}(D_d)$, where $n+1=d(m+1)$. 
Then, we have 
\[M=M_{\theta} = GL_{m_1d} \times GL_{m_2d} = \tM,
\]
 where $m_1d + m_2d = n+1$. So, $M'(F) = GL_{m_1}(D_d) \times GL_{m_2}(D_d)$.

\item[(b)] Let $G = SL_{n+1}$. Note that $G'(F)=SL_{m+1}(D_d)$, where $n+1=d(m+1)$. Then, we have
\[
M=M_{\theta} = G \cap (GL_{m_1d} \times GL_{m_2d}) \hookrightarrow GL_{m_1d} \times GL_{m_2d}=\tM,
\]
 where $m_1d + m_2d = n+1$. So, $M'(F) = G'(F) \cap (GL_{m_1}(D_d) \times GL_{m_2}(D_d))$.
  \end{itemize}

\noindent (2) \textbf{$\boldsymbol{B_n}$ cases}

\[
\xy 
\POS (10,0) *\cir<2pt>{} ="a" ,
\POS (20,0) *\cir<2pt>{} ="b" ,
\POS (30,0) *\cir<0pt>{} ="c" ,
\POS (40,0) *\cir<0pt>{} ="d" ,
\POS (50,0) *\cir<2pt>{} ="e" ,
\POS (60,0) *\cir<2pt>{} ="f",
\POS (70,0) *{\bullet} ="g"

\POS "a" \ar@{-}^<<<{}_<<{}  "b",
\POS "b" \ar@{-}^<<<{}_<<{}  "c",
\POS "c" \ar@{.}^<<<{}_<<{}  "d",
\POS "d" \ar@{-}^<<<{}_<<{}  "e",
\POS "e" \ar@{-}^<<<{}_<<{}  "f",
\POS "f" \ar@{=>}^<{\alpha_{n-1}}_<<<<<<{} "g",
\endxy
\]

Set $\theta = \Delta - \lbrace \alpha_{n-1} \rbrace$, where $\alpha_{n-1} = e_{n-1}-e_n$. 
\begin{itemize}
\item[(a)] Let $G = Spin_{2n+1}$. Then, we have \[
M=M_{\theta} \s GL_n \times SL_2 \hookrightarrow \tM = GL_n \times GL_2.
\]
 So, $M'(F) \s GL_n(F) \times SL_1(D_2)$.

\item[(b)] Let $G = GSpin_{2n+1}$. Then, we have 
\[
M=M_{\theta} \s GL_n \times GL_2=\tM.
\]
 So, $M'(F) \s GL_n(F) \times GL_1(D_2)$.
\end{itemize}

\noindent (3) \textbf{$\boldsymbol{C_{n}}$ cases}

($\mathbf{n}$\textbf{: even})

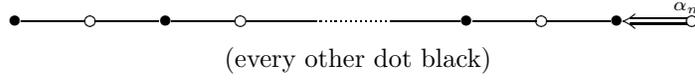
\[
\xy 
\POS (10,0) *{\bullet} ="a" ,
\POS (20,0) *\cir<2pt>{} ="b" ,
\POS (30,0) *{\bullet} ="c" ,
\POS (40,0) *\cir<2pt>{} ="d" ,
\POS (50,0) *\cir<0pt>{} ="e",
\POS (60,0) *\cir<0pt>{} ="f",
\POS (70,0) *{\bullet} ="g",
\POS (80,0) *\cir<2pt>{} ="h",
\POS (90,0) *{\bullet} ="i",
\POS (100,0) *\cir<2pt>{} ="j",

\POS "a" \ar@{-}^<<<{}_<<{}  "b",
\POS "b" \ar@{-}^<<<{}_<<{}  "c",
\POS "c" \ar@{-}^<<<{}_<<{}  "d",
\POS "d" \ar@{-}^<<<{}_<<{}  "e",
\POS "e" \ar@{.}^<{}_<<<{} "f",
\POS "f" \ar@{-}^<{}_<<<{} "g",
\POS "g" \ar@{-}^<{}_<<<{} "h",
\POS "h" \ar@{-}^<{}_<<<{} "i",
\POS "i" \ar@{<=}^>{\alpha_{n}}^>>>{} "j",

\endxy
\]
\begin{center} 
(every other dot black)
\end{center}

Set $\theta = \Delta - \lbrace \alpha_n \rbrace$, where $\alpha_n = 2e_n$.
\begin{itemize}
\item[(a)] Let $\G = Sp_{2n}$. Then, we have 
\[
M = M_{\theta}  \s GL_n = \tM,
\]
 which is the Siegel Levi subgroup. So, $M'(F) \s GL_{n/2}(D_2)$.

\item[(b)] Let $\bG = GSp_{2n}$. Then, we have 
\[
M=M_{\theta} \s GL_n \times GL_1  = \tM.
\]
 So, $M'(F) \s GL_{n/2}(D_2) \times GL_1(F)$.
\end{itemize}

($\mathbf{n}$\textbf{: odd})
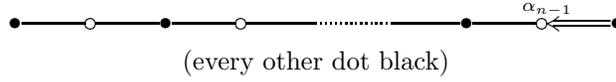
\[
\xy 
\POS (10,0) *{\bullet} ="a" ,
\POS (20,0) *\cir<2pt>{} ="b" ,
\POS (30,0) *{\bullet} ="c" ,
\POS (40,0) *\cir<2pt>{} ="d" ,
\POS (50,0) *\cir<0pt>{} ="e" ,
\POS (60,0) *\cir<0pt>{} ="f",
\POS (70,0) *{\bullet} ="g",
\POS (80,0) *\cir<2pt>{} ="h",
\POS (90,0) *{\bullet} ="i",

\POS "a" \ar@{-}^<<<{}_<<{}  "b",
\POS "b" \ar@{-}^<<<{}_<<{}  "c",
\POS "c" \ar@{-}^<<<{}_<<{}  "d",
\POS "d" \ar@{-}^<<<{}_<<{}  "e",
\POS "e" \ar@{.}^<<<{}_<<{}  "f",
\POS "f" \ar@{-}^<{}_<<<{} "g",
\POS "g" \ar@{-}^<{}_<<<{} "h",
\POS "h" \ar@{<=}^<{\alpha_{n-1}}_<<<{} "i",
\endxy
\]
\begin{center} 
(every other dot black)
\end{center}

Set $\theta = \Delta - \lbrace \alpha_{n-1} \rbrace$, where $\alpha_{n-1}=e_{n-1}-e_n$.
\begin{itemize}
\item[(c)] Let $G = Sp_{2n}$. Then, we have 
\[
M=M_{\theta} \s GL_{n-1} \times SL_2 \hookrightarrow GL_{n-1} \times GL_2 = \tM.
\]
 So, $M'(F) \s GL_{(n-1)/2}(D_2) \times SL_1(D_2)$.

\item[(d)] Let $G = GSp_{2n}$. Then, we have 
\[
M=M_{\theta} \s GL_n \times GL_2 = \tM.
\]
 So, $M'(F) \s GL_{(n-1)/2}(D_2) \times GL_1(D_2)$.
\end{itemize}

\noindent (4) \textbf{$\boldsymbol{D_{n}}$ cases}

($\mathbf{D_n -1}$)

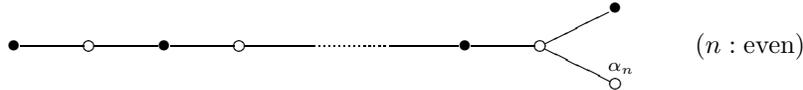
\[
\xy
\POS (10,0) *{\bullet} ="a" ,
\POS (20,0) *\cir<2pt>{} ="b" ,
\POS (30,0) *{\bullet} ="c" ,
\POS (40,0) *\cir<2pt>{} ="d" ,
\POS (50,0) *\cir<0pt>{} ="e",
\POS (60,0) *\cir<0pt>{} ="f",
\POS (70,0) *{\bullet} ="g",
\POS (80,0) *\cir<2pt>{} ="h",
\POS (90,5) *{\bullet} ="i",
\POS (90,-5) *\cir<2pt>{} ="j",

\POS "a" \ar@{-}^<<<{}_<<{}  "b",
\POS "b" \ar@{-}^<<<{}_<<{}  "c",
\POS "c" \ar@{-}^<<<{}_<<{}  "d",
\POS "d" \ar@{-}^<<<{}_<<{}  "e",
\POS "e" \ar@{.}^<<<{}_<<{}  "f",
\POS "f" \ar@{-}^<{}_<<<{} "g",
\POS "g" \ar@{-}^<{}_<<<{} "h",
\POS "h" \ar@{-}^<<{}_<<<{} "i",
\POS "h" \ar@{-}^>{\alpha_{n}}_>>{} "j",
\endxy
\qquad (n: \text{even})
\]

Set $\theta = \Delta - \lbrace \alpha_n \rbrace$, where $\alpha_n = e_{n-1}+e_n$.
\begin{itemize}
\item[(a)] Let $G = Spin_{2n}$. From \cite{kim05, sh88} we have 
\[
M_{\der} = SL_n 
\hookrightarrow 
M=M_{\theta}
\hookrightarrow
 GL_1 \times GL_n = \tM.
\]
So, $M'(F)
\hookrightarrow
 GL_1(F) \times GL_{n/2}(D_2) = \tM'(F).
$

\item[(b)] Let $G = GSpin_{2n}$. Then, we have 
\[
M=M_{\theta} \s GL_1 \times GL_n  = \tM.
\]
 So, $M'(F) \s GL_1 \times GL_{n/2}(D_2)$. 

\item[(c)] Let $G = SO_{2n}$. Then, we have 
\[
M = M_{\theta}  \s GL_n = \tM,
\]
 which is the Siegel Levi subgroup. So, $M'(F) \s GL_{n/2}(D_2)$.
\end{itemize}

($\mathbf{D_n -2}$)

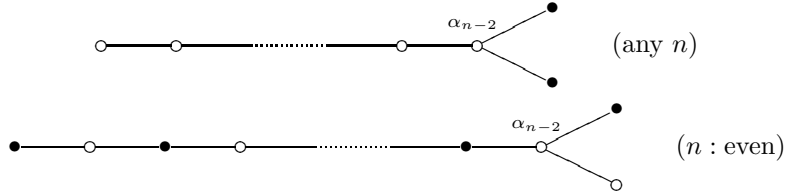
\[
\xy
\POS (10,0) *\cir<2pt>{} ="a" ,
\POS (20,0) *\cir<2pt>{} ="b" ,
\POS (30,0) *\cir<0pt>{} ="c" ,
\POS (40,0) *\cir<0pt>{} ="d" ,
\POS (50,0) *\cir<2pt>{} ="e",
\POS (60,0) *\cir<2pt>{} ="f",
\POS (70,5) *{\bullet} ="g",
\POS (70,-5) *{\bullet} ="h",

\POS "a" \ar@{-}^<<<{}_<<{}  "b",
\POS "b" \ar@{-}^<<<{}_<<{}  "c",
\POS "c" \ar@{.}^<<<{}_<<{}  "d",
\POS "d" \ar@{-}^<<<{}_<<{}  "e",
\POS "e" \ar@{-}^<<<{}_<<{} "f",
\POS "f" \ar@{-}^<<{\alpha_{n-2}}_<<<{} "g",
\POS "f" \ar@{-}^<<<{}_<<{} "h",
\endxy
\qquad (\text{any}~n)
\]

\[
\xy
\POS (10,0) *{\bullet} ="a" ,
\POS (20,0) *\cir<2pt>{} ="b" ,
\POS (30,0) *{\bullet} ="c" ,
\POS (40,0) *\cir<2pt>{} ="d" ,
\POS (50,0) *\cir<0pt>{} ="e",
\POS (60,0) *\cir<0pt>{} ="f",
\POS (70,0) *{\bullet} ="g",
\POS (80,0) *\cir<2pt>{} ="h",
\POS (90,5) *{\bullet} ="i",
\POS (90,-5) *\cir<2pt>{} ="j",

\POS "a" \ar@{-}^<<<{}_<<{}  "b",
\POS "b" \ar@{-}^<<<{}_<<{}  "c",
\POS "c" \ar@{-}^<<<{}_<<{}  "d",
\POS "d" \ar@{-}^<<<{}_<<{}  "e",
\POS "e" \ar@{.}^<<<{}_<<{}  "f",
\POS "f" \ar@{-}^<{}_<<<{} "g",
\POS "g" \ar@{-}^<{}_<<<{} "h",
\POS "h" \ar@{-}^<<{\alpha_{n-2}}_<<<{} "i",
\POS "h" \ar@{-}^>{}_>>{} "j",
\endxy
\qquad (n: \text{even})
\]

Set $\theta = \Delta - \lbrace \alpha_{n-2} \rbrace$, where $\alpha_{n-2} = e_{n-2}-e_{n-1}$. Note that, for each $M$ of type $D_n -2$, there are two inequivalent $F$-inner forms $M'$.
\begin{itemize}
\item[(d)] Let $\bG = Spin_{2n}$. From \cite{kim05, sh88} we have 
\begin{align*}
M_{\der} \s SL_{n-2}\times SL_2 \times SL_2
 &  \hookrightarrow M=M_{\theta} \\
& \hookrightarrow
GL_1 \times GL_{n-2} \times GL_2 \times GL_2 = \tM.
\end{align*}
So, $M'(F)
\hookrightarrow
GL_1(F) \times GL_{n-2}(F) \times GL_1(D_2) \times GL_1(D_2) = \tM'(F)$ for the upper diagram $(\text{any} ~ n)$; $M'(F)
\hookrightarrow GL_1(F) \times GL_{n-2}(F) \times GL_1(D_2) \times GL_2(F) = \tM'(F)$ for the lower diagram $(n:\text{even})$.

\item[(e)] Let $\bG = GSpin_{2n}$. Then, we have 
\[
M=M_{\theta} \s GL_{n-2} \times GL_2 \times GL_2 = \tM.
\]
 So, $M'(F)
\s
GL_1(F) \times GL_{n-2}(F) \times GL_1(D_2) \times GL_1(D_2)$ for the upper diagram $(\text{any} ~ n)$; $M'(F)
\s GL_1(F) \times GL_{n-2}(F) \times GL_1(D_2) \times GL_2(F)$ for the lower diagram $(n:\text{even})$.
\end{itemize}

($\mathbf{D_n -3}$)

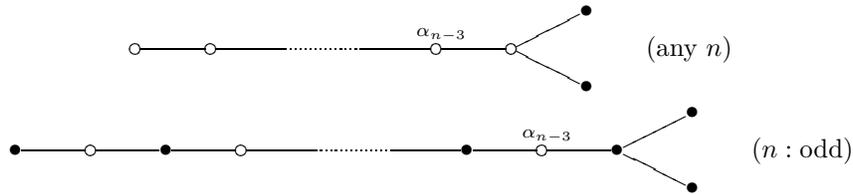
\[
\xy
\POS (10,0) *\cir<2pt>{} ="a" ,
\POS (20,0) *\cir<2pt>{} ="b" ,
\POS (30,0) *\cir<0pt>{} ="c" ,
\POS (40,0) *\cir<0pt>{} ="d" ,
\POS (50,0) *\cir<2pt>{} ="e",
\POS (60,0) *\cir<2pt>{} ="f",
\POS (70,5) *{\bullet} ="g",
\POS (70,-5) *{\bullet} ="h",

\POS "a" \ar@{-}^<<<{}_<<{}  "b",
\POS "b" \ar@{-}^<<<{}_<<{}  "c",
\POS "c" \ar@{.}^<<<{}_<<{}  "d",
\POS "d" \ar@{-}^<<<{}_<<{}  "e",
\POS "e" \ar@{-}^<{\alpha_{n-3}}_<<<{} "f",
\POS "f" \ar@{-}^<<<{}_<<{} "g",
\POS "f" \ar@{-}^<<<{}_<<{} "h",
\endxy
\qquad (\text{any}~n)
\]

\[
\xy
\POS (10,0) *{\bullet} ="a" ,
\POS (20,0) *\cir<2pt>{} ="b" ,
\POS (30,0) *{\bullet} ="c" ,
\POS (40,0) *\cir<2pt>{} ="d" ,
\POS (50,0) *\cir<0pt>{} ="e",
\POS (60,0) *\cir<0pt>{} ="f",
\POS (70,0) *{\bullet} ="g",
\POS (80,0) *\cir<2pt>{} ="h",
\POS (90,0) *{\bullet} ="i",
\POS (100,5) *{\bullet} ="j",
\POS (100,-5) *{\bullet} ="k",

\POS "a" \ar@{-}^<<<{}_<<{}  "b",
\POS "b" \ar@{-}^<<<{}_<<{}  "c",
\POS "c" \ar@{-}^<<<{}_<<{}  "d",
\POS "d" \ar@{-}^<<<{}_<<{}  "e",
\POS "e" \ar@{.}^<<<{}_<<{}  "f",
\POS "f" \ar@{-}^<<<{}_<<{} "g",
\POS "g" \ar@{-}^<<<{}_<<{} "h",
\POS "h" \ar@{-}^<{\alpha_{n-3}}_<<<{} "i",
\POS "i" \ar@{-}^<<<{}_<<{} "j",
\POS "i" \ar@{-}^<<<{}_<<{} "k",
\endxy
\qquad (n: \text{odd})
\]

Set $\theta = \Delta - \lbrace \alpha_{n-3} \rbrace$, where $\alpha_{n-3} = e_{n-3}-e_{n-2}$.
\begin{itemize}
\item[(f)] Let $G = Spin_{2n}$. From \cite{kim05, sh88} we have
\[
M_{\der} \s SL_{n-3} \times SL_4
\hookrightarrow 
M=M_{\theta}
\hookrightarrow
GL_1 \times GL_{n-3} \times GL_4 = \tM.
\]
So, $M'(F)
\hookrightarrow
GL_1(F) \times GL_{n-3}(F) \times GL_2(D_2) = \tM'(F)$ 
for the upper diagram $(\text{any} ~ n)$;
$M'(F)
\hookrightarrow
GL_1(F) \times GL_{(n-3)/2}(D_2) \times GL_1(D_4) = \tM'(F)$ for the lower diagram $(n:\text{odd})$.
\item[(g)] Let $G = GSpin_{2n}$. Then, we have
\[ 
M=M_{\theta}
\s 
GL_{n-2} \times GL_4 = \tM.
\]
So, $M'(F) \s GL_{n-2}(F) \times GL_2(D_2)$ for the upper diagram $(\text{any} ~ n)$;
$M'(F) \s GL_{(n-2)/2}(D_2) \times GL_1(D_4)$ for the lower diagram $(n:\text{odd})$.
\end{itemize}

\noindent (5) $\boldsymbol{E_{6}}$ cases

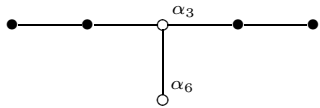
\[
\xy
\POS (0,0) *{\bullet} ="a" ,
\POS (10,0) *{\bullet} ="b" ,
\POS (20,-10) *\cir<2pt>{} ="c" ,
\POS (20,0) *\cir<2pt>{} ="d" ,
\POS (30,0) *{\bullet} ="e" ,
\POS (40,0) *{\bullet} ="f" ,
\POS (50,0)  ="h", 
\POS "a" \ar@{-}^<<<{}_<<{}  "b",
\POS "b" \ar@{-}^<<<{}_<<{}  "d",
\POS "c" \ar@{-}_<<{\alpha_{6}}^<<{}  "d",
\POS "d" \ar@{-}^<<<{\alpha_{3}}_<<{}  "e",
\POS "e" \ar@{-}^<<<{}_<<{}  "f",
\POS "f" \ar@{}^<<<{}_<<{}  "h",
\endxy
\]
\begin{itemize}
\item[(a)] ($\mathbf{E_6 -1}$) Let $G$ be a simply connected group of type $E_6$. Set $\theta = \Delta - \lbrace \alpha_3  \rbrace$, where $\alpha_3 = e_3-e_4$. From \cite{kim05, sh88} we have
\begin{align*}
M_{\der} \s SL_3 \times SL_3  \times SL_2 
& \hookrightarrow 
M=M_{\theta}  \\
& \hookrightarrow
GL_1 \times GL_3 \times GL_3 \times GL_2 = \tM.
\end{align*}
So, $M'(F) \hookrightarrow
GL_1 \times GL_1(D_3) \times GL_1(D_3) \times GL_2(F) = \tM'(F) $.

\item[(b)] (\textbf{(x) in \cite{la67}}) Let $G$ be a simply connected group of type $E_6$. Set $\theta = \Delta - \lbrace \alpha_6 \rbrace$, where $\alpha_6 = e_4+e_5+e_6 +\epsilon$. From \cite{kim05, sh88} we have
\[
M_{\der} \s SL_6 
\hookrightarrow 
M=M_{\theta}
\hookrightarrow
GL_1 \times GL_6 = \tM.
\]
So, $M'(F) \hookrightarrow
GL_1(F) \times GL_2(D_2) = \tM'(F) $.

\item[(c)] Let $G$ be a simply connected group of type $E_6$. Set $\theta = \Delta - \lbrace \alpha_3, \alpha_6 \rbrace$. Then, we have $M_{\der} \s SL_3 \times SL_3 
\hookrightarrow 
M=M_{\theta}
\hookrightarrow
GL_1 \times GL_3 \times GL_3 = \tM $ and $M'(F) \hookrightarrow
GL_1(F) \times GL_1(D_3) \times GL_1(D_3) = \tM'(F) $.
\\

\noindent Note that (a), (b) and (c) above are all possible types of $F$-Levi subgroups of $M'$.
\end{itemize}

\noindent (6) \textbf{$\boldsymbol{E_{7}}$ cases}

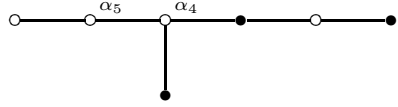
\[
\xy 
\POS (10,0) *\cir<2pt>{} ="b" ,
\POS (30,-10) *{\bullet} ="c" ,
\POS (20,0) *\cir<2pt>{} ="d" ,
\POS (30,0) *\cir<2pt>{} ="e" ,
\POS (40,0) *{\bullet} ="f" ,
\POS (50,0) *\cir<2pt>{} ="h",
\POS (60,0) ="i",
\POS (70,0) ="j",
\POS (60,0) *{\bullet} ="i"
\POS "b" \ar@{-}^<<<{}_<<{}  "d",
\POS "c" \ar@{-}^<<<{}_<<{}  "e",
\POS "d" \ar@{-}^<<<{\alpha_{5}}_<<{}  "e",
\POS "e" \ar@{-}^<<<{\alpha_{4}}_<<{}  "f",
\POS "f" \ar@{-}^<<<{}_<<{}  "h",
\POS "h" \ar@{}^<<<{}_<<{}  "j",
\POS "h" \ar@{-}^<<<{}_<<{} "i",
\endxy
\]

\begin{itemize}
\item[(a)] ($\mathbf{E_7 -1}$) Let $G$ be a simply connected group of type $E_7$. Set $\theta = \Delta - \lbrace \alpha_4  \rbrace$, where $\alpha_4 = e_4-e_5$. From \cite{kim05, sh88} we have
\begin{align*}
M_{\der} \s SL_2 \times SL_3  \times SL_4 
& \hookrightarrow 
M=M_{\theta} \\
& \hookrightarrow
GL_1 \times GL_2 \times GL_3 \times GL_4 = \tM.
\end{align*}
So, $M(F) \hookrightarrow
GL_1(F) \times GL_1(D_2) \times GL_3(F) \times GL_2(D_2) = \tM'(F)$.
\item[(b)] ($\mathbf{E_7 -4}$) Let $G$ be a simply connected group of type $E_7$. Set $\theta = \Delta - \lbrace \alpha_5  \rbrace$, where $\alpha_5 = e_5-e_6$. From \cite{kim05, sh88} we have
\[
M_{\der} \s SL_6 \times SL_2 
\hookrightarrow 
M=M_{\theta}
\hookrightarrow
GL_1 \times GL_6 \times GL_2 = \tM.
\]
So, $M(F) \hookrightarrow GL_1(F) \times GL_3(D_2) \times GL_2(F) = \tM'(F)$.
\end{itemize}

\noindent (7) \textbf{$\boldsymbol{E_8}$, $\boldsymbol{F_{4}}$ and $\boldsymbol{G_2}$ cases}

Any connected reductive algebraic $F$-group $G$ of type $E_8$, $F_{4}$, or $G_2$ does not have non quasi-split $F$-inner forms of $G$ (see Proposition \ref{kot iso}). 

\end{document}